\documentclass[a4paper,12pt]{amsart}

\usepackage{amsfonts}
\usepackage{amsmath}
\usepackage{amssymb}
\usepackage{graphicx}

\usepackage[usenames]{color}
\usepackage[colorlinks]{hyperref}

\newtheorem{thm}{Theorem}[section]
\newtheorem{lem}[thm]{Lemma}
\newtheorem{prop}[thm]{Proposition}

\theoremstyle{definition}

\theoremstyle{remark}
\newtheorem{rem}{Remark}[section]
\newtheorem{defn}{Definition}

\numberwithin{equation}{section}

\def\d{\mathrm d}

\begin{document}

\title[Fractional wave equations with irregular coefficients and data]{Non-homogeneous problem for the fractional wave equation with irregular coefficients and data}

\author[M. Bouguenna]{Manel Bouguenna}
\address{
  Manel Bouguenna:
  \endgraf
  Mustapha Stambouli University of Mascara, 29000 Mascara
  \endgraf
  Algeria
  \endgraf
  {\it E-mail address} {\rm manelbouguenna2001@gmail.com}
}

\author[M. E. Sebih]{Mohammed Elamine Sebih}
\address{
  Mohammed Elamine Sebih:
  \endgraf
  Laboratory of Geomatics, Ecology and Environment (LGEO2E)
  \endgraf
  Mustapha Stambouli University of Mascara, 29000 Mascara
  \endgraf
  Algeria
  \endgraf
  {\it E-mail address} {\rm sebihmed@gmail.com, ma.sebih@univ-mascara.dz}
}

%\thanks{The authors would like to thank Prof. Niyaz Tokmagambetov for his valuable comments.}


\begin{abstract}
In this paper, we consider the Cauchy problem for a non-homogeneous wave equation generated by the fractional Laplacian and involving different kinds of lower order terms. We allow the equation coefficients and data to be of distributional type or less regular, having in mind the Dirac delta function and its powers, and we prove that the problem is well-posed in the sense of the concept of very weak solutions. Moreover, we prove the uniqueness in an appropriate sense and the coherence of the very weak solution concept with classical theory. The results obtained here extend those in \cite{DGL23}. 
\end{abstract}

\maketitle

%%%%%%%%%%%%%%%%%%%%%%%%%%%%%%%%%%%%%%%%%%%%%%%%%%%%%%%%%%%%%%%%%%%%%%%%%%%%%%%%%%%%%%%%%%%%%%%%%%%%%%%%%%%%%%%%%%%%%%%%%%%%%%%%%%%%%%%%%%%%%%%%%%%%%%%%%%%%%%%%%%%%%%%%%%%%%%%%%%%%%%%%%%%%%%%%%%%%

\section{Introduction}
For fixed $T>0$, we consider the Cauchy problem:
\begin{equation}\label{Equation intro}
    \bigg\lbrace
    \begin{array}{l}
    u_{tt}(t,x) + D_{g}^{s}u(t,x) +  m(x)u(t,x) + b(x)u_{t}(t,x)=f(t,x),~~~ (t,x)\in [0,T]\times\mathbb{R}^d \\
    u(0,x)=u_{0}(x),\quad u_{t}(0,x)=u_{1}(x),
    \end{array}
\end{equation}
where $s>0$ and $D_{g}^{s}$ is the differential operator generated by the fractional Laplacian, defined by
\begin{equation*}
     D_{g}^{s}:=(-\Delta)^{\frac{s}{2}}\Big(g(x)(-\Delta)^{\frac{s}{2}}\Big).
\end{equation*}
We assume that:
\begin{itemize}
    \item $g$ is positive,
    \item $m,b$ are non-negative,
    \item $g,m,b,u_0,u_1\in \mathcal{E'}(\mathbb{R}^d)$ (space of distributions with compact support) and for $t\in [0,T]$, $f(t,\cdot)\in \mathcal{E'}(\mathbb{R}^d)$ (i.e. $f$ is smooth enough with respect to $t$ and irregular with respect to $x$).
\end{itemize}
We should highlight here that we understand the first and second assumptions in the sense that the regularisations of $g$ and $m,b$ are positive and non-negative, respectively. Our aim is to prove that the Cauchy problem \eqref{Equation intro} is well posed in the sense of the concept of very weak solutions introduced in \cite{GR15} by Garetto and Ruzhansky in order to give a net solution to the problem of multiplication that Schwartz theory of distributions is concerned with, see \cite{Sch54}, and to provide a framework in which partial differential equations involving coefficients and data of low regularity can be rigorously studied. Let us give a brief literature review about this concept of solutions. After the pioneer work of Garetto and Ruzhansky \cite{GR15}, many researchers started using this notion of solutions for different situations, either for abstract mathematical problems as \cite{CRT21,CRT22a,CRT22b} or for physical models as in \cite{RT17a,RT17b,Gar21} and \cite{MRT19,ART19,ARST21a,ARST21b,ARST21c} where it is shown that the concept of very weak solutions is very suitable for numerical modelling and in \cite{SW22} where the question of propagation of coefficients singularities on the very weak solution is studied. More recently, we cite \cite{GLO21,BLO22,RSY22,RY22,CDRT23,RST24}.

The novelty in this work lies in the fact that we consider equations that can not be formulated in the classical sense. We employ the concept of very weak solutions 
where nonlinear operations make sense for a large collection of singular objects and allows to overcome the problem of the impossibility of multiplication of distributions. Furthermore, the results obtained in this paper extend those of \cite{DGL23} by considering the problem in higher space dimension, whereas in \cite{DGL23} it was treated in the real line. Moreover our approach here is different from the one in \cite{DGL23}.

\section{Preliminaries}
To start with, let us define some notions and specific notations that are used throughout this paper.
\subsection{Notation}
\begin{itemize}
    \item By the notation $f\lesssim g$, we mean that there exists a positive constant $C$, such that $f \leq Cg$ independently on $f$ and $g$.
    \item We also define
    \begin{equation*}
        \Vert u(t,\cdot)\Vert_1 := \Vert u(t,\cdot)\Vert_{L^2} + \Vert (-\Delta)^{\frac{s}{2}}u(t,\cdot)\Vert_{L^2} + \Vert u_t(t,\cdot)\Vert_{L^2},
        \end{equation*}
        and
         \begin{align}
        \Vert u(t,\cdot)\Vert_2 &:= \Vert u(t,\cdot)\Vert_{L^2} + \Vert (-\Delta)^{\frac{s}{2}}u(t,\cdot)\Vert_{L^2} +\Vert (-\Delta)^{s}u(t,\cdot)\Vert_{L^{2}}+ \Vert u_t(t,\cdot)\Vert_{L^2}\nonumber \\ &\quad+\Vert (-\Delta)^{\frac{s}{2}}u_t(t,\cdot)\Vert_{L^2},\nonumber
    \end{align}
      and
         \begin{equation*}
        \Vert u(t,\cdot)\Vert_3 := \Vert u(t,\cdot)\Vert_{L^2} + \Vert (-\Delta)^{\frac{s}{2}}u(t,\cdot)\Vert_{L^2} + \Vert u_t(t,\cdot)\Vert_{L^2} +\Vert (-\Delta)^{\frac{s}{2}}u_t(t,\cdot)\Vert_{L^{2}} ,
    \end{equation*}
     for $t\in (0,T]$.   
\end{itemize}

\subsection{H\"older's inequality}
\begin{prop}\label{Holder inequality}
    Let $r,p,q \geq 1$, such that: $\frac{1}{r}=\frac{1}{p} + \frac{1}{q}$. Assume that $f\in L^{p}(\mathbb{R}^d)$ and $g\in L^{q}(\mathbb{R}^d)$, then, $fg\in L^{r}(\mathbb{R}^d)$ and we have
    \begin{equation}
        \Vert fg\Vert_{L^r}\leq \Vert f\Vert_{L^p}\Vert g\Vert_{L^q}.
    \end{equation}
\end{prop}

\subsection{The fractional Sobolev space $H^s$ and the fractional Laplacian}
\begin{defn}[\textbf{Fractional Sobolev space}]\label{Def. frac. sobolev space}
    Given $s>0$, the fractional Sobolev space is defined by
    \begin{equation*}
        H^s(\mathbb{R}^d) = \big\{f\in L^2(\mathbb{R}^d) : \int_{\mathbb{R}^d}(1+\vert \xi\vert^{2s})|\hat{f}(\xi)|^2\d \xi < +\infty\big\},
    \end{equation*}
    where $\hat{f}$ denotes the Fourier transform of $f$.
\end{defn}
We note that the fractional Sobolev space $H^s$ endowed with the norm
    \begin{equation}\label{Norm H^s}
        \Vert f\Vert_{H^s}:=\bigg(\int_{\mathbb{R}^d}(1+\vert \xi\vert^{2s})|\hat{f}(\xi)|^2\d \xi\bigg)^{\frac{1}{2}},\quad\text{for } f\in H^s(\mathbb{R}^d),
    \end{equation}
    is a Hilbert space.
\begin{defn}[\textbf{Fractional Laplacian}]\label{Def. fract. laplacian}
    For $s>0$, $(-\Delta)^s$ denotes the fractional Laplacian defined by
    \begin{equation*}
    (-\Delta)^{s}f(x) = \mathcal{F}^{-1}(\vert\xi\vert^{2s}(\hat{f}))(x),
    \end{equation*}
    for all $x\in \mathbb{R}^d$.
\end{defn}
In other words, the fractional Laplacian $(-\Delta)^s$ can be viewed as the pseudo-differential operator with the symbol $\vert \xi\vert^{2s}$. With this definition and Plancherel Theorem, the fractional Sobolev space can be defined as:
    \begin{equation}
        H^{s}(\mathbb{R}^{d})=\big\{ f\in L^{2}(\mathbb{R}^{d}): (-\Delta)^{\frac{s}{2}}f \in L^{2}(\mathbb{R}^{d})\big\},
    \end{equation}
    moreover, the norm
    \begin{equation}
        \Vert f\Vert_{H^{s}}:=\Vert f\Vert_{L^2}+\Vert (-\Delta)^{\frac{s}{2}}f\Vert_{L^2},
    \end{equation}
    is equivalent to the one defined in \eqref{Norm H^s}.

\begin{rem}
    We note that the fractional Sobolev space $H^s(\mathbb{R}^d)$ can also be defined via the Gagliardo norm, however, we choosed this approach, since it is valid for any real $s>0$, unlike the one via Gagliardo norm which is valid only for $s\in (0,1)$. We refer the reader to \cite{DPV12,Gar18,Kwa17} for more details and alternative definitions.
\end{rem}
In the following proposition we prove some useful properties of the fractional Laplacian.

\begin{prop}\label{Frac. Lapl. properties}
Let $s>0$ and $f,g\in H^{2s} (\mathbb R^d)$, we have
\begin{itemize}
    \item[(1)] $(-\Delta)^s(f+g)=(-\Delta)^sf+(-\Delta)^sg$.
    \item[(2)] $\langle (-\Delta)^s f,g\rangle_{L^2}=\langle f,(-\Delta)^s g \rangle_{L^2}$.
    \item[(3)] For $s_1,s_2>0$ such that $s_1+s_2=s$, we have $$(-\Delta)^{s_1}\big( (-\Delta)^{s_2} f\big)=(-\Delta)^{s_1+s_2}f.$$
\end{itemize}
\end{prop}

\begin{proof}
    \begin{itemize}
        \item[(1)] The first property follows from the linearity of the Fourier transform. Indeed, we have
        \begin{align}
            \widehat{(-\Delta^s)(f+g)(\cdot)}(\xi)&=\vert \xi \vert^{2s}(\widehat{f(\cdot)+g(\cdot)})(\xi)\nonumber \\ &=\vert \xi \vert^{2s}\hat{f}(\xi)+\vert \xi \vert^{2s}\hat{g}(\xi)\nonumber \\ &=\widehat{(-\Delta^s)(f)(\cdot)}(\xi)+\widehat{(-\Delta^s)(g)(\cdot)}(\xi).\nonumber
        \end{align}
        \item[(2)] By using Plancherel-Parseval identity, we have
        \begin{align}
            \langle (-\Delta)^s f( \cdot),g(\cdot)\rangle_{L^2}&= \langle \vert \xi \vert^{2s} \hat{f}(\xi ),\hat{g}(\xi) \rangle_{L^2}\nonumber \\ 
            &=\int_{\mathbb R^d} \vert \xi \vert^{2s}\hat{f}(\xi )\hat{g}(\xi )d\xi \nonumber \\ &=\int_{\mathbb R^d} \hat{f}(\xi )\vert \xi \vert^{2s}\hat{g}(\xi )d\xi \nonumber \\ &=\langle \hat{f}(\xi),\vert \xi \vert^{2s}\hat{g}(\xi)\rangle_{L^2}\nonumber \\ &=\langle f(\cdot),(-\Delta)^s g(\cdot)\rangle_{L^2}\nonumber. 
        \end{align}
        \item[(3)] Again, by Plancherel-Parseval identity and the self-adjointness of the fractional Laplacian (property (2)), we have 
        \begin{align}
            \langle (-\Delta)^{s_1+s_2}f(\cdot),g(\cdot)\rangle_{L^2}&=\langle \vert \xi \vert^{2(s_1+s_2)}\hat{f}(\xi),\hat{g}(\xi)\rangle_{L^2}\nonumber \\ &=\int_{\mathbb R^d} \vert \xi \vert^{2s_1+2s_2} \hat{f}(\xi )\hat{g}(\xi)d\xi \nonumber \\ &=\int_{\mathbb R^d} \vert \xi \vert^{2s_1}\hat{f}(\xi) \vert \xi \vert^{2s_2}\hat{g}(\xi)d\xi \nonumber \\ &=\langle\vert \xi \vert^{2s_1}\hat{f}(\xi),\vert \xi \vert^{2s_2}\hat{g}(\xi)\rangle_{L^2}\nonumber \\ &= \langle (-\Delta)^{s_1}f(\cdot),(-\Delta)^{s_2}g(\cdot)\rangle_{L^2}\nonumber \\ &=\langle (-\Delta)^{s_2}\big( (-\Delta)^{s_1}f(\cdot)\big),g(\cdot)\rangle_{L^2}.\nonumber
        \end{align}
        Thus: $(-\Delta)^{s_2}\big( (-\Delta)^{s_1}f\big)=(-\Delta)^{s_1+s_2}f$.
    \end{itemize}
\end{proof}
The following inequality will be used frequently in this paper. In the literature, it is known as the fractional Sobolev inequality.

\begin{prop}[\textbf{Fractional Sobolev inequality}, Theorem 1.1. \cite{CT04}]\label{Prop. Sobolev estimate}
    For $d\in \mathbb{N}$ and $s\in \mathbb{R}_+$, let $d>2s$ and $q=\frac{2d}{d-2s}$. Then, the estimate
    \begin{equation}\label{Sobolev estimate}
        \Vert f\Vert_{L^q}^2 \leq C(d,s)\Vert (-\Delta)^{\frac{s}{2}}f\Vert_{L^2}^2,
    \end{equation}
    holds for all $f\in H^{s}(\mathbb{R}^d)$, where the constant $C$ depends only on the dimension $d$ and the order $s$.
\end{prop}

\subsection{Fractional Leibniz rule and the Kato-Ponce inequality} 
For the classical Laplacian, we have the following known identity: 
\begin{equation*}
    \Delta (uv)=\Delta u\,\, v+2\Delta u\,\, \Delta v+u\, \Delta v.
\end{equation*}
For the fractional Laplacian with $s\in(0,1)$, we have an analogous identity: 
\begin{equation}
    (-\Delta)^{s}(uv)=u(-\Delta)^{s}v+v(-\Delta)^{s}u-I_s(u,v),
\end{equation}
where the remainder term is given by:
\begin{equation*}
    I_s(u,v)(x):= C_{d,s}\int_{\mathbb{R}^d}\frac{(u(x)-u(y))(v(x)-v(y))}{\vert x-y\vert ^{d+2s}}dy,\, \, x\in \mathbb{R}^d,
\end{equation*}
and
\begin{equation*}
    \int_{\mathbb{R}^d}\frac{\vert (u(x)-u(y))(v(x)-v(y))\vert }{\vert x-y\vert ^{d+2s}}dy<\infty.
\end{equation*}
\begin{prop}[\textbf{The Kato-Ponce inequality}, Theorem 1. \cite{GO14}]
\label{The Kato-Ponce Inequality}
   Let $\frac{1}{2}<r<\infty$ and $1<p_1,p_2,q_1,q_2\leq \infty$ satisfying $\frac{1}{r}=\frac{1}{p_1}+\frac{1}{q_1}=\frac{1}{p_2}+\frac{1}{q_2}$. Given $s>max(0,\frac{d}{r}-d)$ or $s\in 2\mathbb{N}$, there exists $C=C(d,s,r,p_1,q_1,p_2,q_2)<\infty$, such  that for all $f,h\in \mathcal{S}(\mathbb{R}^d)$ we have
   \begin{equation}\label{The Kato-Ponce Inequality 1}
       \Vert (-\Delta)^{\frac{s}{2}}(fh)\Vert_{{L^r}(\mathbb{R}^d)}\leq C\big[\Vert (-\Delta)^{\frac{s}{2}} f\Vert_{L^{p_1}(\mathbb{R}^d)}\Vert h\Vert_{L^{q_1}(\mathbb{R}^d)}+\Vert f\Vert_{L^{p_2}(\mathbb{R}^d)}\Vert (-\Delta)^{\frac{s}{2}} h\Vert_{L^{q_2}(\mathbb{R}^d)}\big].
   \end{equation}
   \end{prop}

\subsection{Duhamel's principle}
Duhamel's principle is the main tool in proving existence and uniqueness of very weak solutions to our considered problem as well as the coherence with classical solutions. For convenience of the reader, we prove the following special version of this principle and refer the reader to \cite{ER18} for more details and applications.

Let us consider the following Cauchy problem
\begin{equation}
    \left\lbrace
    \begin{array}{l}
    u_{tt}(t,x)+\lambda(x)u_{t}(t,x)+Lu(t,x)=f(t,x) ,~~~(t,x)\in\left(0,\infty\right[\times \mathbb{R}^{d},\\
    u(0,x)=u_{0}(x),\\
    u_{t}(0,x)=u_{1}(x), \,\,\, x\in\mathbb{R}^{d}, \label{Equation Duhamel 1}
    \end{array}
    \right.
\end{equation}
for a given function $\lambda$ and $L$ is a linear partial differential operator acting over the spatial variable.

\begin{thm}\label{Thm Duhamel}
The solution to (\ref{Equation Duhamel 1}) is given by
\begin{equation}\label{Sol Duhamel}
    u(t,x)= w(t,x) + \int_0^t v(t,x;\tau)\d \tau,
\end{equation}
where $w(t,x)$ is the solution to the homogeneous problem
\begin{equation}\label{Homog eqn Duhamel}
    \left\lbrace
    \begin{array}{l}
    w_{tt}(t,x)+\lambda(x)w_{t}(t,x)+Lw(t,x)=0 ,~~~(t,x)\in\left(0,\infty\right)\times \mathbb{R}^{d},\\
    w(0,x)=u_{0}(x),\,\,\, w_{t}(0,x)=u_{1}(x),\,\,\, x\in\mathbb{R}^{d},
    \end{array}
    \right.
\end{equation}
and $v(t,x;\tau)$ solves the auxiliary Cauchy problem
\begin{equation}\label{Aux eqn Duhamel}
    \left\lbrace
    \begin{array}{l}
    v_{tt}(t,x;\tau)+\lambda(x)v_{t}(t,x;\tau)+Lv(t,x;\tau)=0 ,~~~(t,x)\in\left(\tau,\infty\right)\times \mathbb{R}^{d},\\
    v(\tau,x;\tau)=0,\,\,\, v_{t}(\tau,x;\tau)=f(\tau,x),\,\,\, x\in\mathbb{R}^{d},
    \end{array}
    \right.
\end{equation}
where $\tau$ is a parameter varying over $\left(0,\infty\right)$.
\end{thm}
\begin{proof}
    Firstly, we apply $\partial_{t}$ to $u$ in \eqref{Sol Duhamel}. We get
    \begin{equation}\label{sol Duhamel u_t}
        \partial_t u(t,x)=\partial_t w(t,x) + \int_0^t \partial_t v(t,x;\tau)\d \tau,
    \end{equation}
    and accordingly
    \begin{equation}\label{sol Duhamel lambdau_t}
        \lambda(x)\partial_t u(t,x)=\lambda(x)\partial_t w(t,x) + \int_0^t \lambda(x)\partial_t v(t,x;\tau)\d \tau,
    \end{equation}
    where we used the fact that $v(t,x;t)=0$ by the imposed initial condition in \eqref{Aux eqn Duhamel}. We differentiate again \eqref{sol Duhamel u_t} with respect to $t$ to get
    \begin{equation}\label{sol Duhamel u_tt}
        \partial_{tt} u(t,x)=\partial_{tt} w(t,x) + f(t,x) + \int_0^t \partial_{tt} v(t,x;\tau)\d \tau,
    \end{equation}
    where we used that $\partial_t v(t,x;t) = f(t,x)$. Now, applying $L$ to $u$ in \eqref{Sol Duhamel} gives
    \begin{equation}\label{sol Duhamel L}
        L u(t,x)=L w(t,x) + \int_0^t L v(t,x;\tau)\d \tau.
    \end{equation}
    By adding \eqref{sol Duhamel u_tt}, \eqref{sol Duhamel L} and \eqref{sol Duhamel lambdau_t} with in mind that $w$ and $v$ satisfy the equations in \eqref{Homog eqn Duhamel} and \eqref{Aux eqn Duhamel}, we get
    \begin{equation*}
        u_{tt}(t,x)+\lambda(x)u_{t}(t,x)+Lu(t,x)=f(t,x).
    \end{equation*}
    It remains to prove that $u$ satisfies the initial conditions. Indeed, from \eqref{Sol Duhamel} and \eqref{sol Duhamel u_t}, we have that $u(0,x)=w(0,x)=u_0(x)$ and that $u_t(0,x)=\partial_t w(0,x)=u_1(x)$. This concludes the proof.    
\end{proof}

\begin{rem}
    We note that the above statement of Duhamel's principle can be extended to differential operators of order $k\in \mathbb{N}$. Indeed, if we consider the Cauchy problem
    \begin{equation*}\label{Equation Duhamel general}
        \left\lbrace
        \begin{array}{l}        \partial_{t}^{k}u(t,x)+\sum_{j=1}^{k-1}\lambda_{j}(x)\partial_{t}^{j} u(t,x)+Lu(t,x)=f(t,x) ,~~~(t,x)\in\left(0,\infty\right)\times \mathbb{R}^{d},\\
        \partial_{t}^{j}u(0,x)=u_{j}(x),\,\text{for}\quad j=0,\cdots,k-1, \quad x\in\mathbb{R}^{d},
        \end{array}
        \right.
    \end{equation*}
    then, the solution is given by
\begin{equation*}\label{Sol Duhamel general}
    u(t,x)= w(t,x) + \int_0^t v(t,\tau;\tau)\d \tau,
\end{equation*}
where $w(t,x)$ is the solution to the homogeneous problem
\begin{equation*}\label{Homog eqn Duhamel general}
    \left\lbrace
    \begin{array}{l}
    \partial_{t}^{k}w(t,x)+\sum_{j=1}^{k-1}\lambda_{j}(x)\partial_{t}^{j}w(t,x)+Lw(t,x)=0 ,~~~(t,x)\in\left(0,\infty\right)\times \mathbb{R}^{d},\\
    \partial_{t}^{j}w(0,x)=u_{j}(x),\,\text{for}\quad j=0,\cdots,k-1, \quad x\in\mathbb{R}^{d},
    \end{array}
    \right.
\end{equation*}
and $v(t,x;\tau)$ solves the auxiliary Cauchy problem
\begin{equation*}\label{Aux eqn Duhamel general}
    \left\lbrace
    \begin{array}{l}
    \partial_{t}^{k}v(t,x;\tau)+\sum_{j=1}^{k-1}\lambda_{j}(x)\partial_{t}^{j}v_{t}(t,x;\tau)+Lv(t,x;\tau)=0 ,~~~(t,x)\in\left(\tau,\infty\right)\times \mathbb{R}^{d},\\
    \partial_{t}^{j}w(\tau,x;\tau)=0,\,\text{for}\quad j=0,\cdots,k-2, \,\,\, \partial_{t}^{k-1}w(\tau,x;\tau)=f(\tau,x),\,\,\, x\in\mathbb{R}^{d},
    \end{array}
    \right.
\end{equation*}
where $\tau\in \left(0,\infty\right)$.
\end{rem}

\section{Energy estimates for the classical solution}
In this section, we prove energy estimates to the Cauchy problem \eqref{Equation intro} in the case when the equation coefficients and the initial data are regular enough in such a way that a classical solution exists. We prove estimates for different classes of regularity of the coefficients $g,m,b,f$ and the initial data $u_0$ and $u_1$. These estimates are a key element when proving the existence and uniqueness of a very weak solution to the Cauchy problem \eqref{Equation intro} as well as the coherence with classical solutions.

\subsection{Auxiliary result}
Before proving energy estimates for the classical solution to \eqref{Equation intro}, the following auxiliary result for the Cauchy problem
\begin{equation}\label{Equation aux}
    \bigg\lbrace
    \begin{array}{l}
    u_{tt}(t,x) + D_{g}^{s}u(t,x) = 0,\quad (t,x)\in [0,T]\times\mathbb{R}^d,\\
    u(0,x)=u_{0}(x),\quad u_{t}(0,x)=u_{1}(x),\quad x\in \mathbb{R}^d,
    \end{array}
\end{equation}
needs to be proved.

\begin{lem}\label{Aux. result}
    Let $g\in L^{\infty}(\mathbb{R}^d)$ be positive and assume that $u_0\in H^{s}(\mathbb R^d)$ and $u_1\in L^2(\mathbb R^d)$. Then, there is a unique solution $u\in C([0,T];H^{s}(\mathbb{R}^d))\cap C^1([0,T];L^{2}(\mathbb{R}^d))$ to the Cauchy problem \eqref{Equation aux} that satisfies the estimate
    \begin{equation}\label{Energy aux}
        \Vert u(t,\cdot)\Vert_1\lesssim \Big(1 + \Vert g\Vert_{L^{\infty}}^{\frac{1}{2}}\Big) \Big[\Vert u_0\Vert_{H^s} + \Vert u_{1}\Vert_{L^2}\Big],
    \end{equation}
    uniformly in $t\in [0,T]$.

Moreover, let $d>2s$, then for $g\in L^{\frac{d}{2s}}(\mathbb{R}^d)$ and $u_0\in H^{2s}(\mathbb R^d)$, $u_1\in L^2(\mathbb R^d)$, there is a unique solution $u\in C\big([0,T]; H^{s}(\mathbb{R}^d)\cap C^{1}\big([0,T]; L^{2}(\mathbb{R}^d)\big)$ to \eqref{Equation aux} that satisfies the estimate
    \begin{equation}\label{Energy1 aux}
        \Vert u(t,\cdot)\Vert_{1} \lesssim \Big(1+\Vert g\Vert_{L^{\frac{d}{2s}}}^{\frac{1}{2}}\Big)\Big[\Vert u_{0}\Vert_{H^{2s}}+\Vert u_{1}\Vert_{L^2}\Big],
    \end{equation}
     uniformly in $t\in [0,T]$.
\end{lem}

\begin{proof}
    We multiply the equation in \eqref{Equation aux} by $u_t$, we integrate with respect to the variable $x$ over $\mathbb{R}^d$ and we take the real part to get
    \begin{equation}\label{Energy functional aux.}
        Re\Big(\langle u_{tt}(t,\cdot),u_{t}(t,\cdot)\rangle_{L^2} + \langle D_{g}^{s}u(t,\cdot),u_{t}(t,\cdot)\rangle_{L^2}\Big) = 0.
    \end{equation}
    We have
    \begin{equation}\label{En1 aux.}
        Re\langle u_{tt}(t,\cdot),u_{t}(t,\cdot)\rangle_{L^2} = \frac{1}{2}\partial_{t}\langle u_{t}(t,\cdot),u_{t}(t,\cdot)\rangle_{L^2} = \frac{1}{2}\partial_{t}\Vert u_{t}(t,\cdot)\Vert_{L^2}^2,
    \end{equation}
    and by using the self-adjointness of the fractional Laplacian (see Proposition \ref{Frac. Lapl. properties}), we get
    \begin{align}\label{En2 aux.}
        Re\langle D_{g}^{s}u(t,\cdot),u_{t}(t,\cdot)\rangle_{L^2} &= \frac{1}{2}\partial_{t}\langle g^{\frac{1}{2}}(\cdot)(-\Delta)^{\frac{s}{2}}u(t,\cdot),g^{\frac{1}{2}}(\cdot)(-\Delta)^{\frac{s}{2}}u(t,\cdot)\rangle_{L^2} \\
        & = \frac{1}{2}\partial_{t}\Vert g^{\frac{1}{2}}(\cdot)(-\Delta)^{\frac{s}{2}}u(t,\cdot)\Vert_{L^2}^2. \nonumber
    \end{align}
    It follows by substituting \eqref{En1 aux.} and \eqref{En2 aux.} into \eqref{Energy functional aux.} that
    \begin{equation}\label{Energy functional1 aux.}
        \partial_{t}\Big[ \Vert u_{t}(t,\cdot)\Vert_{L^2}^2 + \Vert g^{\frac{1}{2}}(\cdot)(-\Delta)^{\frac{s}{2}}u(t,\cdot)\Vert_{L^2}^2 \Big] = 0,
    \end{equation}
    this leads to the energy conservation law
    \begin{equation}\label{En3 aux.}
        \Vert u_{t}(t,\cdot)\Vert_{L^2}^2 + \Vert g^{\frac{1}{2}}(\cdot)(-\Delta)^{\frac{s}{2}}u(t,\cdot)\Vert_{L^2}^2 = \Vert u_{1}\Vert_{L^2}^2 + \Vert g^{\frac{1}{2}}(\cdot)(-\Delta)^{\frac{s}{2}}u_0\Vert_{L^2}^2.
    \end{equation}
    Thus, we get
    \begin{align}\label{Energy1 aux1}
        \Vert u_{t}(t,\cdot)\Vert_{L^2}^2 &\leq \Vert u_{1}\Vert_{L^2}^2 + \Vert g^{\frac{1}{2}}(\cdot)(-\Delta)^{\frac{s}{2}}u_0\Vert_{L^2}^2\\
        & \leq \Vert u_{1}\Vert_{L^2}^2 + \Vert g\Vert_{L^{\infty}}\Vert (-\Delta)^{\frac{s}{2}}u_0\Vert_{L^2}^2 \nonumber\\
        & \leq \Vert u_{1}\Vert_{L^2}^2 + \Vert g\Vert_{L^{\infty}}\Vert u_0\Vert_{H^s}^2,\nonumber
    \end{align}
    and similarly
    \begin{equation}\label{Energy2 aux.}
        \Vert g^{\frac{1}{2}}(\cdot)(-\Delta)^{\frac{s}{2}}u(t,\cdot)\Vert_{L^2}^2 \leq \Vert u_{1}\Vert_{L^2}^2 + \Vert g\Vert_{L^{\infty}}\Vert u_0\Vert_{H^s}^2.
    \end{equation}
    Using the assumption that $g$ is bounded from below, that is 
$$
\inf\limits_{x\in \mathbb{R}^d}g(x)=c_0 >0,
$$ 
we get
    \begin{equation}\label{Energy3 aux.}
        \Vert (-\Delta)^{\frac{s}{2}}u(t,\cdot)\Vert_{L^2}^2 \leq \Vert u_{1}\Vert_{L^2}^2 + \Vert g\Vert_{L^{\infty}}\Vert u_0\Vert_{H^s}^2.
    \end{equation}
    In order to estimate the solution $u$, we make use of the fundamental theorem of calculus. We have 
    \begin{equation}\label{Fund th of calc}
        u(t,x) = u_0 + \int_{0}^{t}u_{t}(\tau,x)\d \tau.
    \end{equation}
    Thus, by taking the $L^2$ norm in \eqref{Fund th of calc} and using Minkowski's integral inequality, together with \eqref{Energy1 aux1} and that $t\in [0,T]$, we get
    \begin{align}
        \Vert u(t,\cdot)\Vert_{L^2} &\leq \Vert u_0\Vert_{L^2} + \int_{0}^{t}\Vert u_{t}(\tau,\cdot)\Vert_{L^2}\d \tau\\
        & \leq \Vert u_0\Vert_{L^2} + \Vert u_{1}\Vert_{L^2} + \Vert g\Vert_{L^{\infty}}^{\frac{1}{2}}\Vert u_0\Vert_{H^s}\nonumber\\
        &\lesssim \Big(1 + \Vert g\Vert_{L^{\infty}}^{\frac{1}{2}}\Big) \Big[\Vert u_0\Vert_{H^s} + \Vert u_{1}\Vert_{L^2}\Big].\nonumber
    \end{align}
    This completes the proof of the first part of the lemma. Now, let $d>2s$. Arguing as above, we get the same conservation law formula:
    \begin{align}\label{En3 aux. q neq 2}
        \Vert u_{t}(t,\cdot)\Vert_{L^2}^2 + \Vert g^{\frac{1}{2}}(\cdot)(-\Delta)^{\frac{s}{2}}u(t,\cdot)\Vert_{L^2}^2 &\leq \Vert u_{1}\Vert_{L^2}^2 + \Vert g^{\frac{1}{2}}(\cdot)(-\Delta)^{\frac{s}{2}}u_0\Vert_{L^2}^2\\ &\leq \Vert u_{1}\Vert_{L^2}^2 + \Vert g^{\frac{1}{2}}\Vert_{L^p}^2\Vert(-\Delta)^{\frac{s}{2}}u_0\Vert_{L^q}^2\nonumber,
    \end{align}
    where the second term on the right hand side is estimated by using H\"older's inequality (see Theorem \ref{Holder inequality}) for $1<p,q<\infty$, satisfying $\frac{1}{p}+\frac{1}{q}=\frac{1}{2}$. Now, if we choose $q=\frac{2d}{d-2s}$ and consequently $p=\frac{d}{s}$, it follows from Proposition \ref{Prop. Sobolev estimate} that
\begin{equation}
\Vert u_0\Vert_{L^q} \lesssim \Vert (-\Delta)^{\frac{s}{2}}u_{0}(\cdot)\Vert_{L^2}\leq \Vert u_{0}\Vert_{H^s}.
\end{equation}
Thus
    \begin{align}\label{En4 aux. q neq 2}
        \Vert u_{t}(t,\cdot)\Vert_{L^2}^2 + \Vert g^{\frac{1}{2}}(\cdot)(-\Delta)^{\frac{s}{2}}u(t,\cdot)\Vert_{L^2}^2 &\lesssim \Vert u_{1}\Vert_{L^2}^2 +  \Vert g\Vert_{L^{\frac{p}{2}}} \Vert (-\Delta)^{\frac{s}{2}}\big ((-\Delta)^{\frac{s}{2}}u_0\big )\Vert_{L^2}^2 \\
        & \lesssim \Vert u_{1}\Vert_{L^2}^2 +  \Vert g\Vert_{L^{\frac{p}{2}}} \Vert (-\Delta)^{s}u_0\Vert_{L^2}^2 \\
        & \lesssim \Vert u_{1}\Vert_{L^2}^2 + \Vert g\Vert_{L^{\frac{d}{2s}}}\Vert u_0\Vert_{H^{2s}}^2,\nonumber
    \end{align}
    where we used that $\Vert g^{\frac{1}{2}}\Vert_{L^p}^2 = \Vert g\Vert_{L^{\frac{p}{2}}}$. To estimate $u$, we use again the fundamental theorem of calculus to finally get
    \begin{align} 
        \Vert u(t,\cdot)\Vert_{L^2} &\leq \Vert u_0\Vert_{L^2} + \int_{0}^{t}\Vert u_{t}(\tau,\cdot)\Vert_{L^2}\d \tau\\
        & \lesssim \Vert u_0\Vert_{H^{2s}} + \Vert u_{1}\Vert_{L^2} + \Vert g\Vert_{L^{\frac{d}{2s}}}^{\frac{1}{2}}\Vert u_0\Vert_{H^{2s}}\nonumber \\
        & \lesssim \Big(1+\Vert g\Vert_{L^{\frac{d}{2s}}}^{\frac{1}{2}}\Big)\Big[\Vert u_{0}\Vert_{H^{2s}}+\Vert u_{1}\Vert_{L^2}\Big],\nonumber
    \end{align}
ending the proof.    
\end{proof}

\subsection{Homogeneous equation case}
Let us now consider the Cauchy problem \eqref{Equation intro} in the case when the source term $f\equiv 0$. In this case, we have the following result.
\begin{lem}\label{Lemma1}
    Let $g\in L^{\infty}(\mathbb{R}^d)$ be positive and $m,b\in L^{\infty}(\mathbb{R}^d)$ be non-negative. Suppose that $u_0 \in H^{s}(\mathbb{R}^d)$ and $u_1 \in L^{2}(\mathbb{R}^d)$. Then the unique solution $u\in C([0,T];H^{s}(\mathbb{R}^d))\cap C^1([0,T];L^{2}(\mathbb{R}^d))$ to the Cauchy problem \eqref{Equation intro} satisfies the estimate
    \begin{equation}\label{Energy estimate}
         \Vert u(t,\cdot)\Vert_1 \lesssim \Big(1 + \Vert g \Vert_{L^{\infty}} + \Vert m\Vert_{L^{\infty}}\Big) \Big(1+ \Vert b \Vert_{L^{\infty}} +  \Vert m \Vert_{L^{\infty}}^{\frac{1}{2}}\Big) \Big[ \Vert u_0\Vert_{H^s} + \Vert u_1\Vert_{L^2} \Big]  ,
    \end{equation}
     for all $t\in [0,T]$.

Moreover, if $g\equiv 1$, there is a unique solution $u\in C([0,T];H^{s}(\mathbb{R}^d))\cap  C^1([0,T];L^{2}(\mathbb{R}^d))$ to \eqref{Equation intro}, and it satisfies
    \begin{equation}\label{Energy estimate g1}
         \Vert u(t,\cdot)\Vert_1 \lesssim \big(2+ \Vert m\Vert_{L^{\infty}}\big) \big(1+ \Vert b \Vert_{L^{\infty}} +  \Vert m \Vert_{L^{\infty}}^{\frac{1}{2}}\big) \Big[ \Vert u_0\Vert_{H^s} + \Vert u_1\Vert_{L^2} \Big]  ,
    \end{equation}
     for all $t\in [0,T]$.
     
     Let $d>2s$. Then, for $g\in \big(L^{\frac{d}{2s}}(\mathbb{R}^d)\cap L^{\frac{d}{s}}(\mathbb{R}^d)\cap W^{s,\frac{d}{s}}(\mathbb{R}^d)\big)$, $m\in \big(L^{\frac{d}{2s}}(\mathbb{R}^d)\cap L^{\frac{d}{s}}(\mathbb{R}^d)\big)$, $b\in L^{\frac{d}{s}}(\mathbb{R}^d)$, and $u_0 \in H^{3s}(\mathbb{R}^d)$, $u_1 \in H^{2s}(\mathbb{R}^d)$, there is a unique solution $u\in C\big([0,T]; H^{s}(\mathbb{R}^d)\cap C^{1}\big([0,T]; H^{s}(\mathbb{R}^d)\big)$ to \eqref{Equation intro} that satisfies the estimate
     \begin{align}\label{Energy1 aux1+1}
        \Vert u(t,\cdot)\Vert_3 &\lesssim \big( 1+\Vert m\Vert_{L^{\frac{d}{s}}}+\Vert b\Vert_{L^{\frac{d}{s}}} \big ) \bigg[ 1+ \Vert m\Vert_{L^{\frac{d}{s}}} +\Vert b\Vert_{L^{\frac{d}{s}}} +\Vert g\Vert_{L^{\frac{d}{2s}}}^{\frac{1}{2}}   \\ &\quad+\Vert m\Vert_{L^{\frac{d}{2s}}}^{\frac{1}{2}} +\Vert b\Vert_{L^{\frac{d}{s}}} \Vert g\Vert_{W^{s,\frac{d}{s}}} \bigg]\Big[ \Vert u_0\Vert_{H^{3s}} +\Vert u_1\Vert_{H^{2s}}\Big] \nonumber,
    \end{align}
   uniformly in $t\in [0,T]$.

   Moreover, if $g\equiv 1$, and $u_0\in H^{2s}(\mathbb{R}^d)$, $u_1\in H^{s}(\mathbb{R}^d)$, then, the unique solution $u\in C([0,T]; H^{2s}(\mathbb{R}^d))\cap C^{1}([0,T]; H^{s}(\mathbb{R}^d))$ to \eqref{Equation intro}  satisfies
   \begin{equation}\label{Energy estimate1}
         \Vert u(t,\cdot)\Vert_2 \lesssim \Big(1+\Vert m\Vert_{L^{\frac{d}{s}}}\Big)\Big(1+\Vert m\Vert_{L^{\frac{d}{2s}}}\Big)\Big(1+\Vert b\Vert_{L^{\frac{d}{s}}}\Big)^2\Big[\Vert u_0\Vert_{H^{2s}} + \Vert u_{1}\Vert_{H^s}\Big],
    \end{equation}
     uniformly in $t\in [0,T]$.
\end{lem}

\begin{proof}
    Multiplying the equation in \eqref{Equation intro} by $u_t$ and integrating with respect to the variable $x$ over $\mathbb{R}^d$ and taking the real part, we get
    \begin{align} \label{Energy functional}
        Re\Bigg(\langle u_{tt}(t,\cdot),&u_{t}(t,\cdot)\rangle_{L^2} + \langle D_{g}^{s}u(t,\cdot),u_{t}(t,\cdot)\rangle_{L^2}\\
        & + \langle m(\cdot)u(t,\cdot),u_{t}(t,\cdot)\rangle_{L^2} + \langle b(\cdot)u_{t}(t,\cdot),u_{t}(t,\cdot)\rangle_{L^2}\Bigg) = 0.\nonumber
    \end{align}
    We easily show that
    \begin{equation}\label{En1}
        Re\langle u_{tt}(t,\cdot),u_{t}(t,\cdot)\rangle_{L^2} = \frac{1}{2}\partial_{t}\langle u_{t}(t,\cdot),u_{t}(t,\cdot)\rangle_{L^2} = \frac{1}{2}\partial_{t}\Vert u_{t}(t,\cdot)\Vert_{L^2}^2,
    \end{equation}
    and
    \begin{align}\label{En2}
        Re\langle D_{g}^{s} u(t,\cdot),u_{t}(t,\cdot)\rangle_{L^2} &= \frac{1}{2}\partial_{t}\langle g^{\frac{1}{2}}(\cdot)(-\Delta)^{\frac{s}{2}}u(t,\cdot),g^{\frac{1}{2}}(\cdot)(-\Delta)^{\frac{s}{2}}u(t,\cdot)\rangle_{L^2} \\
        & = \frac{1}{2}\partial_{t}\Vert g^{\frac{1}{2}}(\cdot)(-\Delta)^{\frac{s}{2}}u(t,\cdot)\Vert_{L^2}^2 \nonumber,
    \end{align}
    where we used the self-adjointness of the fractional Laplacian (see Proposition \ref{Frac. Lapl. properties}). For the remaining terms in \eqref{Energy functional}, we proceed as follows
    \begin{align}\label{En3}
        Re\langle m(\cdot)u(t,\cdot),u_{t}(t,\cdot)\rangle_{L^2}&=\frac{1}{2}\partial_t \langle m^{\frac{1}{2}}(\cdot) u(t,\cdot), m^{\frac{1}{2}}(\cdot) u(t,\cdot)\rangle_{L^2} \\&=\frac{1}{2}\partial_{t} \Vert m^{\frac{1}{2}} (\cdot) u(t,\cdot)\Vert_{L^2}^2,\nonumber
    \end{align}
    and
    \begin{align}\label{En4}
        Re\langle b(\cdot)u_{t}(t,\cdot),u_{t}(t,\cdot)\rangle_{L^2}&=\frac{1}{2} \langle b^{\frac{1}{2}}(\cdot)u_{t}(t,\cdot),b^{\frac{1}{2}}(\cdot)u_{t}(t,\cdot)\rangle_{L^2} \\ &= \frac{1}{2} \Vert b^{\frac{1}{2}}(\cdot)u_{t}(t,\cdot)\Vert_{L^2}^2.\nonumber
    \end{align}
    By substituting \eqref{En1},\eqref{En2},\eqref{En3} and \eqref{En4} in \eqref{Energy functional} we get
    \begin{equation}\label{Energy functional1}
        \partial_{t}\Big[ \Vert u_{t}(t,\cdot)\Vert_{L^2}^2 + \Vert g^{\frac{1}{2}}(\cdot) (-\Delta)^{\frac{s}{2}}u(t,\cdot)\Vert_{L^2}^2 + \Vert m^{\frac{1}{2}}(\cdot) u(t,\cdot)\Vert_{L^2}^2 \Big] =  -\Vert b^{\frac{1}{2}}(\cdot)u_{t}(t,\cdot)\Vert_{L^2}^2.
    \end{equation}
    Let us define the energy function of the system \eqref{Equation intro} by
    \begin{equation}\label{Energy function}
        E(t):= \Vert u_{t}(t,\cdot)\Vert_{L^2}^2 + \Vert g^{\frac{1}{2}}(\cdot) (-\Delta)^{\frac{s}{2}}u(t,\cdot)\Vert_{L^2}^2 + \Vert m^{\frac{1}{2}}(\cdot)u(t,\cdot)\Vert_{L^2}^2.
    \end{equation}    
    It follows from \eqref{Energy functional1} that $\partial_{t}E(t)\leq 0$ and consequently that we have a decay of energy, that is: $E(t)\leq E(0)$ for all $t\in [0,T]$, where
    \begin{equation}
        E(0):= \Vert u_{1} \Vert_{L^2}^2 + \Vert g^{\frac{1}{2}}(\cdot) (-\Delta)^{\frac{s}{2}}u_{0}\Vert_{L^2}^2 + \Vert m^{\frac{1}{2}}(\cdot)u_{0}\Vert_{L^2}^2.
    \end{equation}
    By taking into consideration the estimate
    \begin{equation}
        \Vert m^{\frac{1}{2}}(\cdot)u_{0}\Vert_{L^2}^2 \leq \Vert m\Vert_{L^{\infty}}\Vert u_0\Vert_{L^2}^2,
    \end{equation}
    it follows that all terms in \eqref{Energy function} satisfy the estimates
    \begin{align}\label{Energy1}
        \Vert m^{\frac{1}{2}}(\cdot)u(t,\cdot)\Vert_{L^2}^2 & \leq \Vert u_{1}\Vert_{L^2}^2 + \Vert g^{\frac{1}{2}} (-\Delta)^{\frac{s}{2}}u_0\Vert_{L^2}^2 + \Vert m\Vert_{L^{\infty}}\Vert u_0\Vert_{L^2}^2\\
        & \lesssim \Vert u_{1}\Vert_{L^2}^2 +\Vert g \Vert_{L^{\infty}}  \Vert u_0\Vert_{H^s}^2 + \Vert m\Vert_{L^{\infty}}\Vert u_0\Vert_{H^s}^2\nonumber\\
        & \lesssim \big(1 + \Vert g \Vert_{L^{\infty}} + \Vert m\Vert_{L^{\infty}}\big)\big[ \Vert u_0\Vert_{H^s} + \Vert u_1\Vert_{L^2} \big]^2 \nonumber ,
    \end{align}
    as well as
    \begin{equation}\label{Energy2}
        \bigg\{\Vert u_{t}(t,\cdot)\Vert_{L^2}^2, \Vert (-\Delta)^{\frac{s}{2}}u(t,\cdot)\Vert_{L^2}^2\bigg\} \lesssim \big(1 + \Vert g \Vert_{L^{\infty}} + \Vert m\Vert_{L^{\infty}}\big)\big[ \Vert u_0\Vert_{H^s} + \Vert u_1\Vert_{L^2} \big]^2,
    \end{equation}
    uniformly in $t\in [0,T]$, where we used the fact that
    \begin{equation}
        \Big\{ \Vert (-\Delta)^{\frac{s}{2}}u_0\Vert_{L^2}, \Vert u_0\Vert_{L^2}\Big\} \leq \Vert u_0\Vert_{H^s},
    \end{equation}
    and that $g$ is bounded from below. Now, we need to estimate $u$. To this end, we recall the estimate \eqref{Energy aux} in Lemma \ref{Aux. result}, that is    
    \begin{align} \label{Energy aux 1+1}
        \bigg\{\Vert u(t,\cdot)\Vert_{L^2}, \Vert (-\Delta)^{\frac{s}{2}}u(t,\cdot)\Vert_{L^2}, \Vert u_{t}(t,\cdot)\Vert_{L^2}\bigg\} & \leq \Vert u_0\Vert_{L^2} + \Vert u_{1}\Vert_{L^2} + \Vert g\Vert_{L^{\infty}}^{\frac{1}{2}}\Vert u_0\Vert_{H^s}, 
     \end{align} 
    and we apply Duhamel's principle (Theorem \ref{Thm Duhamel} with $\lambda \equiv0$) for the problem
    \begin{equation}\label{Equation aux 1+1}
    \bigg\lbrace
    \begin{array}{l}
    u_{tt}(t,x) + D_{g}^{s}u(t,x) =\delta (t,x) ,\quad (t,x)\in [0,T]\times\mathbb{R}^d,\\
    u(0,x)=u_{0}(x),\quad u_{t}(0,x)=u_{1}(x),\quad x\in \mathbb{R}^d,
    \end{array}
    \end{equation} 
    where $\delta (t,x):=-m(x)u(t,x)-b(x)u_{t}(t,x)$. The solution to \eqref{Equation aux 1+1} is given by
\begin{equation}\label{solution mess and dissipation estimate}
u(t,x)=W(t,x)+\int_{0}^{T}V(t,x;\tau)d\tau,
\end{equation}
 where $W(t,x)$ is the solution to the homogeneous problem
    \begin{equation*}\label{Equation1 mb}
        \bigg\{
        \begin{array}{l}
        W_{tt}(t,x) +D_g^s W(t,x) =0,\quad (t,x)\in [0,T]\times\mathbb{R}^d, \\
        W(0,x)=u_{0}(x),\quad W_{t}(0,x)=u_{1}(x),\quad x\in \mathbb R^d,
        \end{array}
    \end{equation*}
    and $V(t,x;\tau)$ solves the problem
\begin{equation*}
    \left\lbrace
    \begin{array}{l}
    V_{tt}(t,x;\tau)+D_g^s V(t,x;\tau) =0,\,\,\, (t,x)\in\left(\tau,T\right]\times \mathbb{R}^{d},\\
    V(\tau,x;\tau)=0, ~V_{t}(\tau,x;\tau)=\delta(\tau,x).
    \end{array}
    \right.
\end{equation*}
Taking the $L^{2}$-norm in \eqref{solution mess and dissipation estimate} and by using the above mentioned estimate \eqref{Energy aux 1+1}, we get
\begin{equation}
    \Vert u(t,\cdot)\Vert_{L^2}  \leq \Vert W_{\varepsilon}(t,\cdot)\Vert_{L^2} + \int_{0}^{T}\Vert V_{\varepsilon}(t,\cdot;s)\Vert_{L^2} ds,\label{Duhamel1 solution estimate mb}    
\end{equation}
  where we used  Minkowski's integral inequality. 
  We have
    \begin{equation*}
        \Vert W_{\varepsilon}(t,\cdot)\Vert_{L^2} \lesssim\Vert u_0\Vert_{L^2} + \Vert u_{1}\Vert_{L^2} + \Vert g\Vert_{L^{\infty}}^{\frac{1}{2}}\Vert u_0\Vert_{H^s},
    \end{equation*}
    and
    \begin{equation*}
        \Vert V_{\varepsilon}(t,\cdot;s)\Vert_{L^2} \leq  \Vert \delta(\tau,\cdot)\Vert_{L^2},
    \end{equation*} 
    since $V_{\varepsilon}$ satisfies the estimate \eqref{Energy aux 1+1}. Thus, we get
        \begin{equation}\label{Estimate1 u}
        \Vert u(t,\cdot)\Vert_{L^2}\lesssim \Vert u_0\Vert_{L^2} + \Vert u_{1}\Vert_{L^2} + \Vert g\Vert_{L^{\infty}}^{\frac{1}{2}}\Vert u_0\Vert_{H^s} + \int_{0}^{t}\Vert \delta (\tau,\cdot)\Vert_{L^2}d\tau,
    \end{equation}
    for all $t\in [0,T]$.
    To estimate $\Vert \delta (\tau,\cdot)\Vert_{L^2}$, the last term in the above inequality, we use the triangle inequality and the estimates
    \begin{align}
        \Vert m(\cdot)u(\tau,\cdot)\Vert_{L^2} & \leq \Vert m\Vert_{L^{\infty}}^{\frac{1}{2}} \Vert m^{\frac{1}{2}}(\cdot)u(\tau,\cdot)\Vert_{L^2}\\
        & \lesssim \Vert m\Vert_{L^{\infty}}^{\frac{1}{2}}\big(1 + \Vert g \Vert_{L^{\infty}} + \Vert m\Vert_{L^{\infty}}\big)^{\frac{1}{2}}\big[ \Vert u_0\Vert_{H^s} + \Vert u_1\Vert_{L^2} \big] \nonumber\\
        & \lesssim \big(1 + \Vert g \Vert_{L^{\infty}} + \Vert m\Vert_{L^{\infty}}\big)\big[ \Vert u_0\Vert_{H^s} + \Vert u_1\Vert_{L^2} \big], \nonumber
    \end{align}
    resulting from \eqref{Energy1}, and similarly
    \begin{align}
        \Vert b(\cdot)u_{t}(\tau,\cdot)\Vert_{L^2} & \leq \Vert b\Vert_{L^{\infty}} \Vert u_{t}(\tau,\cdot)\Vert_{L^2}\\
        & \lesssim \Vert b\Vert_{L^{\infty}}\big(1 + \Vert g \Vert_{L^{\infty}} + \Vert m\Vert_{L^{\infty}}\big)^{\frac{1}{2}}\big[ \Vert u_0\Vert_{H^s} + \Vert u_1\Vert_{L^2} \big], \nonumber
    \end{align}
    resulting from \eqref{Energy2}, to get
    \begin{equation}\label{Estimate f}
        \Vert \delta (\tau,\cdot)\Vert_{L^2} \lesssim \big( \Vert b \Vert_{L^{\infty}} +  \Vert m \Vert_{L^{\infty}}^{\frac{1}{2}}\big) \big(1 + \Vert g \Vert_{L^{\infty}} + \Vert m\Vert_{L^{\infty}}\big)^{\frac{1}{2}}\big[ \Vert u_0\Vert_{H^s} + \Vert u_1\Vert_{L^2} \big].
    \end{equation}
    The desired estimate for $u$ follows by substituting \eqref{Estimate f} into \eqref{Estimate1 u}. We obtain
    \begin{align}\label{Estimate1 u 1+1}
        \Vert u(t,\cdot)\Vert_{L^2} &\lesssim  \Vert u_0\Vert_{L^2} + \Vert u_{1}\Vert_{L^2} + \Vert g\Vert_{L^{\infty}}^{\frac{1}{2}}\Vert u_0\Vert_{H^s} \\ &\quad + \big( \Vert b \Vert_{L^{\infty}} +  \Vert m \Vert_{L^{\infty}}^{\frac{1}{2}}\big) \big(1 + \Vert g \Vert_{L^{\infty}} + \Vert m\Vert_{L^{\infty}}\big)^{\frac{1}{2}}\big[ \Vert u_0\Vert_{H^s}  + \Vert u_1\Vert_{L^2} \big]  \nonumber \\ &\lesssim \big(1 + \Vert g \Vert_{L^{\infty}} + \Vert m\Vert_{L^{\infty}}\big)^{\frac{1}{2}}\big[ \Vert u_0\Vert_{H^s} + \Vert u_1\Vert_{L^2} \big] \nonumber \\ &\quad + \big( \Vert b \Vert_{L^{\infty}} +  \Vert m \Vert_{L^{\infty}}^{\frac{1}{2}}\big) \big(1 + \Vert g \Vert_{L^{\infty}} + \Vert m\Vert_{L^{\infty}}\big)^{\frac{1}{2}}\big[ \Vert u_0\Vert_{H^s}  + \Vert u_1\Vert_{L^2} \big]  \nonumber \\ & \lesssim \big(1 + \Vert g \Vert_{L^{\infty}} + \Vert m\Vert_{L^{\infty}}\big)^{\frac{1}{2}} \big(1+ \Vert b \Vert_{L^{\infty}} +  \Vert m \Vert_{L^{\infty}}^{\frac{1}{2}}\big) \big[ \Vert u_0\Vert_{H^s} + \Vert u_1\Vert_{L^2} \big] \nonumber \\
       & \lesssim \big(1 + \Vert g \Vert_{L^{\infty}} + \Vert m\Vert_{L^{\infty}}\big) \big(1+ \Vert b \Vert_{L^{\infty}} +  \Vert m \Vert_{L^{\infty}}^{\frac{1}{2}}\big) \big[ \Vert u_0\Vert_{H^s} + \Vert u_1\Vert_{L^2} \big] \nonumber ,
    \end{align}
    where we used that $\lambda^{\frac{1}{2}}\leq \lambda$ for $\lambda \geq 1$. Let us now consider the case when $g\equiv 1$. The desired estimate is a direct consequence of \eqref{Estimate1 u 1+1} by taking that $\Vert g\Vert_{L^{\infty}}=1$, for $g\equiv 1$. We get
    \begin{equation}
        \Vert u(t,\cdot)\Vert_1 \lesssim \big(2+ \Vert m\Vert_{L^{\infty}}\big) \big(1+ \Vert b \Vert_{L^{\infty}} +  \Vert m \Vert_{L^{\infty}}^{\frac{1}{2}}\big) \Big[ \Vert u_0\Vert_{H^s} + \Vert u_1\Vert_{L^2} \Big].
    \end{equation}
    This ends the proof of the first part of the lemma. Now, let $d>2s$. By repeating the above reasoning we get the same conservation law formula
    \begin{align}\label{En3 aux. q neq 2 d}
        \Vert u_{t}(t,\cdot)\Vert_{L^2}^2 + \Vert g^{\frac{1}{2}}(\cdot)&(-\Delta)^{\frac{s}{2}}u(t,\cdot)\Vert_{L^2}^2 + \Vert m^{\frac{1}{2}} (\cdot) u(t,\cdot) \Vert_{L^2}^2 \\ & \leq \Vert u_{1}\Vert_{L^2}^2 +\Vert g^{\frac{1}{2}}(\cdot)(-\Delta)^{\frac{s}{2}}u_0\Vert_{L^2}^2+ \Vert m^{\frac{1}{2}} (\cdot) u_0 \Vert_{L^2}^2\nonumber \\
        & \leq \Vert u_{1}\Vert_{L^2}^2 + \Vert g^{\frac{1}{2}}\Vert_{L^p}^2\Vert(-\Delta)^{\frac{s}{2}}u_0\Vert_{L^q}^2  +\Vert m^{\frac{1}{2}} \Vert_{L^p}^2 \Vert u_0 \Vert_{L^q}^2,\nonumber 
    \end{align}
    where the second and the third terms on the right hand side are estimated by using H\"older's inequality (see Theorem \ref{Holder inequality}) for $1<p<\infty$ such that $\frac{1}{p} + \frac{1}{q} = \frac{1}{2}$. Now, if we choose $q=\frac{2d}{d-2s}$ and consequently $p=\frac{d}{s}$, it follows from Proposition \ref{Prop. Sobolev estimate} that 
    \begin{equation}
      \Vert u_0\Vert_{L^q} \lesssim \Vert (-\Delta)^{\frac{s}{2}}u_{0}(\cdot)\Vert_{L^2}\leq \Vert u_{0}\Vert_{H^s},
    \end{equation}
   and by using the assumption that $g$ is bounded from below, that is 
   $$
   \inf_{x\in \mathbb{R}^d}g(x)=c_0 >0,
   $$ 
   we get   
    \begin{align}\label{En4 aux. q neq 2 d}
        \Vert u_{t}(t,\cdot)\Vert_{L^2}^2 + \Vert (-\Delta)^{\frac{s}{2}}&u(t,\cdot)\Vert_{L^2}^2 +\Vert m^{\frac{1}{2}} (\cdot) u(t,\cdot) \Vert_{L^2}^2 \\ 
         &\lesssim \Vert u_{1}\Vert_{L^2}^2 + \Vert g\Vert_{L^{\frac{p}{2}}}\Vert u_0\Vert_{H^{2s}}^2 + \Vert m\Vert_{L^{\frac{p}{2}}} \Vert u_0 \Vert_{H^s}^2\nonumber \\
        & \lesssim \Vert u_{1}\Vert_{L^2}^2 + \Vert g\Vert_{L^{\frac{d}{2s}}}\Vert u_0\Vert_{H^{2s}}^2 + \Vert m\Vert_{L^{\frac{d}{2s}}} \Vert u_0 \Vert_{H^{2s}}^2\nonumber \\
        &\lesssim \Vert u_{1}\Vert_{L^2}^2 +( \Vert g\Vert_{L^{\frac{d}{2s}}} + \Vert m\Vert_{L^{\frac{d}{2s}}})\Vert u_0\Vert_{H^{2s}}^2, \nonumber
    \end{align}
where we used that $\Vert g^{\frac{1}{2}}\Vert_{L^p}^2 = \Vert g\Vert_{L^{\frac{p}{2}}}$ and the embedding $H^{2s}(\mathbb R^d) \hookrightarrow H^s(\mathbb R^d)$. To estimate $u$,  we recall again the second estimate in Lemma \eqref{Aux. result}, that is
    \begin{equation}\label{Energy1 aux2}
        \bigg\{\Vert u(t,\cdot)\Vert_{L^2}, \Vert (-\Delta)^{\frac{s}{2}}u(t,\cdot)\Vert_{L^2}, \Vert u_{t}(t,\cdot)\Vert_{L^2}\bigg\} \lesssim \Vert u_0\Vert_{L^2} + \Vert u_{1}\Vert_{L^2} + \Vert g\Vert_{L^{\frac{d}{2s}}}^{\frac{1}{2}}\Vert u_0\Vert_{H^{2s}},
    \end{equation}
     and we use Duhamel's principle and we argue similarly as in the first part of the proof, to finally get
    \begin{align} \label{estimate u }
        \Vert u(t,\cdot)\Vert_{L^2} &\leq \Vert u_0\Vert_{L^2} +\Vert u_{1}\Vert_{L^2} + \Vert g\Vert_{L^{\frac{d}{2s}}}^{\frac{1}{2}}\Vert u_0\Vert_{H^{2s}}+\int_{0}^{t}\Vert \delta (\tau,\cdot) \Vert_{L^2}d\tau,
    \end{align}
    for all $t\in [0,T]$, with $\delta (t,x):=-m(x)u(t,x)-b(x)u_{t}(t,x)$. In order to estimate $\Vert \delta (\tau,\cdot)\Vert_{L^2}$, we use the triangle inequality to get
    \begin{equation}\label{Estimate2 f}
        \Vert \delta (\tau,\cdot)\Vert_{L^2}\leq \Vert m(\cdot)u(\tau,\cdot)\Vert_{L^2} + \Vert b(\cdot)u_{t}(\tau,\cdot)\Vert_{L^2}.
    \end{equation}
    To estimate the first term in \eqref{Estimate2 f}, we use H\"older's inequality (Theorem \ref{Holder inequality}) to get
    \begin{equation}
        \Vert m(\cdot)u(\tau,\cdot)\Vert_{L^2} \leq \Vert m\Vert_{L^p} \Vert u(\tau,\cdot)\Vert_{L^q},
    \end{equation}
    for $1<p,q<\infty$, satisfying $\frac{1}{p}+\frac{1}{q}=\frac{1}{2}$, and we choose $q=\frac{2d}{d-2s}$ and consequently $p=\frac{d}{s}$, in order to get (from Proposition \ref{Prop. Sobolev estimate})
    \begin{equation}
        \Vert u(t,\cdot)\Vert_{L^q} \lesssim \Vert (-\Delta)^{\frac{s}{2}}u(t,\cdot)\Vert_{L^2},
    \end{equation}
   and thus
   \begin{equation}\label{Estimate au}
        \Vert m(\cdot)u(\tau,\cdot)\Vert_{L^2} \lesssim \Vert m\Vert_{L^{\frac{d}{s}}} \Vert (-\Delta)^{\frac{s}{2}}u(\tau,\cdot)\Vert_{L^2},
    \end{equation}
    for all $t\in[0,T]$. Using the estimate \eqref{En4 aux. q neq 2 d}, we arrive at
    \begin{align}\label{Estimate2 au}
        \Vert m(\cdot)u(\tau,\cdot)\Vert_{L^2} & \lesssim \Vert m\Vert_{L^{\frac{d}{s}}} \Big[\Vert u_{1}\Vert_{L^2} + \Vert g\Vert_{L^{\frac{d}{2s}}}^{\frac{1}{2}}\Vert u_0\Vert_{H^{2s}} + \Vert m\Vert_{L^{\frac{d}{2s}}}^{\frac{1}{2}} \Vert u_0 \Vert_{H^{2s}} \Big].
    \end{align}
     For the second term in \eqref{Estimate2 f}, we argue as above, to get
    \begin{equation}\label{Estimate bu_t}
     \Vert b(\cdot)u_t(\tau,\cdot)\Vert_{L^2} \lesssim \Vert b\Vert_{L^{\frac{d}{s}}} \Vert (-\Delta)^{\frac{s}{2}}u_t(\tau,\cdot)\Vert_{L^2},
    \end{equation}
    for all $t\in[0,T]$. We need now to estimate $\Vert (-\Delta)^{\frac{s}{2}}u_t(\tau,\cdot)\Vert_{L^2}$. For this, we note that if $u$ solves the Cauchy problem
    \begin{equation}
        \bigg\lbrace
        \begin{array}{l}
        u_{tt}(t,x) + D_g^{s}u(t,x) + m(x)u(t,x) + b(x)u_{t}(t,x)=0, \,\, (t,x)\in [0,T]\times\mathbb{R}^d,\\
        u(0,x)=u_{0}(x),\quad u_{t}(0,x)=u_{1}(x), \quad x\in\mathbb{R}^d,
        \end{array}
    \end{equation}
    then $u_t$ solves
    \begin{equation}\label{Equation in u_t}
        \bigg\lbrace
        \begin{array}{l}
        (u_t)_{tt}(t,x) + D_g^{s}u_t(t,x) + m(x)u_t(t,x) + b(x)(u_t)_{t}(t,x)=0, \,\, (t,x)\in [0,T]\times\mathbb{R}^d,\\
        u_t(0,x)=u_{1}(x),\quad (u_t)_{t}(0,x)=-D_g^{s}u_0(x) - m(x)u_0(x) - b(x)u_1(x), \quad x\in\mathbb{R}^d.
        \end{array}
    \end{equation}
    Thanks to \eqref{Estimate au} and \eqref{Estimate bu_t} one has
    \begin{equation}
        \Vert m(\cdot)u_0(\cdot)\Vert_{L^2} \lesssim \Vert m\Vert_{L^{\frac{d}{s}}} \Vert u_0\Vert_{H^s},\quad \Vert b(\cdot)u_1(\cdot)\Vert_{L^2} \lesssim \Vert b\Vert_{L^{\frac{d}{s}}} \Vert u_1\Vert_{H^s}.
   \end{equation}
    The estimate for $\Vert (-\Delta)^{\frac{s}{2}}u_t(\tau,\cdot)\Vert_{L^2}$ follows by using \eqref{Energy1 aux2} applied to the problem \eqref{Equation in u_t}, to get
    \begin{align}\label{Estimate Delta u_t}
        \Vert (-\Delta)^{\frac{s}{2}}u_t(\tau,\cdot)\Vert_{L^2} & \lesssim \Vert u_{tt}(0,.)\Vert_{L^2} + \Vert g\Vert_{L^{\frac{d}{2s}}}^{\frac{1}{2}}\Vert u_1\Vert_{H^{2s}} + \Vert m\Vert_{L^{\frac{d}{2s}}}^{\frac{1}{2}} \Vert u_1 \Vert_{H^{2s}} \\
        & \lesssim \Vert (-\Delta)^{\frac{s}{2}}\big(g(\cdot)(-\Delta)^{\frac{s}{2}}u_0 \big) \Vert_{L^2} +\Vert m(\cdot) u_0  \Vert_{L^2} + \Vert b (\cdot) u_1  \Vert_{L^2} \nonumber \\ &\quad + \Vert g\Vert_{L^{\frac{d}{2s}}}^{\frac{1}{2}}\Vert u_1\Vert_{H^{2s}} + \Vert m\Vert_{L^{\frac{d}{2s}}}^{\frac{1}{2}} \Vert u_1 \Vert_{H^{2s}}.\nonumber  
    \end{align}
    To estimate the term $\Vert (-\Delta)^{\frac{s}{2}}\big(g(\cdot)(-\Delta)^{\frac{s}{2}}u_0 (\cdot)\big) \Vert_{L^2} $ in \eqref{Estimate Delta u_t}, we first use the Kato-Ponce inequality (see Proposition \ref{The Kato-Ponce Inequality}) applied for $r=2$. We get
    \begin{align}
      \Vert (-\Delta)^{\frac{s}{2}}\big(g(\cdot)&(-\Delta)^{\frac{s}{2}}u_0 (\cdot)\big) \Vert_{L^2} \\ &\lesssim \Vert g\Vert_{L^{p_1}}\Vert(-\Delta)^s u_0\Vert_{L^{q_1}}+\Vert(-\Delta)^{\frac{s}{2}}g(\cdot)\Vert_{L^{p_2}} \Vert (-\Delta)^{\frac{s}{2}}u_0\Vert_{L^{q_2}},\nonumber
    \end{align}
    with $\frac{1}{2}=\frac{1}{p_1}+\frac{1}{q_1}=\frac{1}{p_2}+\frac{1}{q_2}$. Now, we choose $q_1=q_2=\frac{2d}{d-2s}$ and consequently $p_1=p_2=\frac{d}{s}$, to get
    \begin{align}\label{the kato-ponce inequality}
      \Vert (-\Delta)^{\frac{s}{2}}\big(g(\cdot)&(-\Delta)^{\frac{s}{2}}u_0 (\cdot)\big) \Vert_{L^2} \\ &\lesssim \Vert g\Vert_{L^{\frac{d}{s}}}\Vert(-\Delta)^s u_0\Vert_{L^{\frac{2d}{d-2s}}}+\Vert(-\Delta)^{\frac{s}{2}}g(\cdot)\Vert_{L^{\frac{d}{s}}} \Vert (-\Delta)^{\frac{s}{2}}u_0\Vert_{L^{\frac{2d}{d-2s}}}\nonumber  \\ &\lesssim \Vert g\Vert_{L^{\frac{d}{s}}}\Vert(-\Delta)^{\frac{3s}{2}} u_0\Vert_{L^2}+\Vert(-\Delta)^{\frac{s}{2}}g(\cdot)\Vert_{L^{\frac{d}{s}}} \Vert (-\Delta)^s u_0\Vert_{L^2} \nonumber \\
      &\lesssim \Vert g\Vert_{L^{\frac{d}{s}}}\Vert u_0\Vert_{H^{3s}}+\Vert(-\Delta)^{\frac{s}{2}}g(\cdot)\Vert_{L^{\frac{d}{s}}} \Vert  u_0\Vert_{H^{2s}}.\nonumber
    \end{align}
    By substituting \eqref{the kato-ponce inequality} into \eqref{Estimate Delta u_t}, we get
    \begin{align}\label{Estimate Delta u_t 1}
        \Vert (-\Delta)^{\frac{s}{2}}u_t(\tau,\cdot)\Vert_{L^2}  & \lesssim \Vert (-\Delta)^{\frac{s}{2}}\big(g(\cdot)(-\Delta)^{\frac{s}{2}}u_0 \big) \Vert_{L^2} +\Vert m\Vert_{L^{\frac{d}{s}}} \Vert u_0  \Vert_{L^{\frac{2d}{d-2s}}} \\ &\quad + \Vert b \Vert_{L^{\frac{d}{s}}} \Vert u_1 \Vert_{L^{\frac{2d}{d-2s}}}  + \Vert g\Vert_{L^{\frac{d}{2s}}}^{\frac{1}{2}}\Vert u_1\Vert_{H^{2s}} + \Vert m\Vert_{L^{\frac{d}{2s}}}^{\frac{1}{2}} \Vert u_1 \Vert_{H^{2s}} \nonumber \\ & \lesssim \Vert g\Vert_{L^{\frac{d}{s}}}\Vert u_0\Vert_{H^{3s}}+\Vert(-\Delta)^{\frac{s}{2}}g(\cdot)\Vert_{L^{\frac{d}{s}}} \Vert  u_0\Vert_{H^{2s}}+\Vert m\Vert_{L^{\frac{d}{s}}} \Vert u_0  \Vert_{H^s} \nonumber \\ &\quad + \Vert b \Vert_{L^{\frac{d}{s}}} \Vert u_1 \Vert_{H^s} + \Vert g\Vert_{L^{\frac{d}{2s}}}^{\frac{1}{2}}\Vert u_1\Vert_{H^{2s}} + \Vert m\Vert_{L^{\frac{d}{2s}}}^{\frac{1}{2}} \Vert u_1 \Vert_{H^{2s}}, \nonumber 
    \end{align}
and by substituting \eqref{Estimate Delta u_t 1} into \eqref{Estimate bu_t}, we obtain
    \begin{align}\label{Estimate2 bu_t}
        \Vert b(\cdot)u_t(\tau,\cdot)\Vert_{L^2} & \lesssim \Vert b\Vert_{L^{\frac{d}{s}}}\big[\Vert g\Vert_{L^{\frac{d}{s}}}\Vert u_0\Vert_{H^{3s}}+\Vert(-\Delta)^{\frac{s}{2}}g(\cdot)\Vert_{L^{\frac{d}{s}}} \Vert  u_0\Vert_{H^{2s}} \\ &+\Vert m\Vert_{L^{\frac{d}{s}}} \Vert u_0  \Vert_{H^s} + \Vert b \Vert_{L^{\frac{d}{s}}} \Vert u_1 \Vert_{H^s}  + \Vert g\Vert_{L^{\frac{d}{2s}}}^{\frac{1}{2}}\Vert u_1\Vert_{H^{2s}} + \Vert m\Vert_{L^{\frac{d}{2s}}}^{\frac{1}{2}} \Vert u_1 \Vert_{H^{2s}} \big] \nonumber \\ & \lesssim \Vert b\Vert_{L^{\frac{d}{s}}}\big[\Vert g\Vert_{W^{s,\frac{d}{s}}}\Vert u_0\Vert_{H^{3s}}+\Vert m\Vert_{L^{\frac{d}{s}}} \Vert u_0  \Vert_{H^s} + \Vert b \Vert_{L^{\frac{d}{s}}} \Vert u_1 \Vert_{H^s} \nonumber \\ &\quad+ \Vert g\Vert_{L^{\frac{d}{2s}}}^{\frac{1}{2}}\Vert u_1\Vert_{H^{2s}} + \Vert m\Vert_{L^{\frac{d}{2s}}}^{\frac{1}{2}} \Vert u_1 \Vert_{H^{2s}} \big]\nonumber,
    \end{align}
    where we used that:
    $\Vert g\Vert_{W^{s,\frac{d}{s}}}=\Vert g\Vert_{L^{\frac{d}{s}}}+\Vert (-\Delta)^{\frac{s}{2}}g(\cdot)\Vert_{L^{\frac{d}{s}}}$. The estimate for $\Vert \delta (\tau,\cdot)\Vert_{L^2}$ follows from  \eqref{Estimate f} and \eqref{Estimate2 au} with \eqref{Estimate2 bu_t}, yielding
    \begin{align} \label{Estimate3 f}
        \Vert \delta (\tau,\cdot)\Vert_{L^2} &\lesssim\Vert m\Vert_{L^{\frac{d}{s}}} \big[\Vert u_{1}\Vert_{L^2} + \Vert g\Vert_{L^{\frac{d}{2s}}}^{\frac{1}{2}}\Vert u_0\Vert_{H^{2s}} + \Vert m\Vert_{L^{\frac{d}{2s}}}^{\frac{1}{2}} \Vert u_0 \Vert_{H^{2s}}\big] \\ &\quad+\Vert b\Vert_{L^{\frac{d}{s}}}\big[\Vert g\Vert_{W^{s,\frac{d}{s}}}\Vert u_0\Vert_{H^{3s}}+\Vert m\Vert_{L^{\frac{d}{s}}} \Vert u_0  \Vert_{H^s} + \Vert b \Vert_{L^{\frac{d}{s}}} \Vert u_1 \Vert_{H^s} \nonumber \\ &\quad+ \Vert g\Vert_{L^{\frac{d}{2s}}}^{\frac{1}{2}}\Vert u_1\Vert_{H^{2s}} + \Vert m\Vert_{L^{\frac{d}{2s}}}^{\frac{1}{2}} \Vert u_1 \Vert_{H^{2s}} \big] \nonumber  \\
        &\lesssim \Vert m\Vert_{L^{\frac{d}{s}}}  \Vert u_1\Vert_{L^2} +\Vert m\Vert_{L^{\frac{d}{s}}} \Vert g\Vert_{L^{\frac{d}{2s}}}^{\frac{1}{2}} \Vert u_0\Vert_{H^{2s}}+\Vert m\Vert_{L^{\frac{d}{s}}}\Vert m\Vert_{L^{\frac{d}{2s}}}^{\frac{1}{2}} \Vert u_0\Vert_{H^{2s}}  \nonumber \\ &\quad+\Vert b\Vert_{L^{\frac{d}{s}}}  \Vert g\Vert_{W^{s,\frac{d}{s}}} \Vert u_0\Vert_{H^{3s}}+\Vert b\Vert_{L^{\frac{d}{s}}} \Vert m\Vert_{L^{\frac{d}{s}}} \Vert u_0\Vert_{H^{s}} +\Vert b\Vert_{L^{\frac{d}{s}}}^{2}\Vert u_1\Vert_{H^{s}} \nonumber \\ &\quad+\Vert b\Vert_{L^{\frac{d}{s}}}\Vert g\Vert_{L^{\frac{d}{2s}}}^{\frac{1}{2}} \Vert u_1\Vert_{H^{2s}} +\Vert b\Vert_{L^{\frac{d}{s}}} \Vert m\Vert_{L^{\frac{d}{2s}}}^{\frac{1}{2}} \Vert u_1\Vert_{H^{2s}} \nonumber \\
        &\lesssim \big( 1+\Vert m\Vert_{L^{\frac{d}{s}}}+\Vert b\Vert_{L^{\frac{d}{s}}} \big ) \bigg [ \Vert m\Vert_{L^{\frac{d}{s}}} +\Vert b\Vert_{L^{\frac{d}{s}}} +\Vert g\Vert_{L^{\frac{d}{2s}}}^{\frac{1}{2}} +\Vert m\Vert_{L^{\frac{d}{2s}}}^{\frac{1}{2}} \nonumber \\ &\quad+\Vert b\Vert_{L^{\frac{d}{s}}} \Vert g\Vert_{W^{s,\frac{d}{s}}} \bigg] \bigg[ \Vert u_0\Vert_{H^{3s}} +\Vert u_1\Vert_{H^{2s}}\bigg]. \nonumber
    \end{align}
   Our estimate for the solution $u$ follows by substituting \eqref{Estimate3 f} into \eqref{estimate u }. We arrive at
\begin{align} 
        \Vert u(t,\cdot)\Vert_{L^2} &\lesssim \Vert u_0\Vert_{L^2} +\Vert u_{1}\Vert_{L^2} + \Vert g\Vert_{L^{\frac{d}{2s}}}^{\frac{1}{2}}\Vert u_0\Vert_{H^{2s}}+\big( 1+\Vert m\Vert_{L^{\frac{d}{s}}}+\Vert b\Vert_{L^{\frac{d}{s}}} \big ) \\ &\quad \times \big [ \Vert m\Vert_{L^{\frac{d}{s}}} +\Vert b\Vert_{L^{\frac{d}{s}}} +\Vert g\Vert_{L^{\frac{d}{2s}}}^{\frac{1}{2}} +\Vert m\Vert_{L^{\frac{d}{2s}}}^{\frac{1}{2}} +\Vert b\Vert_{L^{\frac{d}{s}}} \Vert g\Vert_{W^{s,\frac{d}{s}}} \big]\nonumber\\ & \quad \times\big[ \Vert u_0\Vert_{H^{3s}} +\Vert u_1\Vert_{H^{2s}}\big] \nonumber \\
        &\lesssim \big( 1+\Vert m\Vert_{L^{\frac{d}{s}}}+\Vert b\Vert_{L^{\frac{d}{s}}} \big ) \bigg [ 1+ \Vert m\Vert_{L^{\frac{d}{s}}} +\Vert b\Vert_{L^{\frac{d}{s}}} +\Vert g\Vert_{L^{\frac{d}{2s}}}^{\frac{1}{2}} +\Vert m\Vert_{L^{\frac{d}{2s}}}^{\frac{1}{2}} \nonumber \\ &+\Vert b\Vert_{L^{\frac{d}{s}}} \Vert g\Vert_{W^{s,\frac{d}{s}}} \bigg] \bigg[ \Vert u_0\Vert_{H^{3s}} +\Vert u_1\Vert_{H^{2s}}\bigg] \nonumber .
    \end{align}
Let us now consider the case when $g\equiv 1$. The desired estimate is a direct consequence of Lemma 2.5 in \cite{RST24}. We recall it here:
\begin{align}\label{Energy estimate1_Lemma2.5}
         \big\{\Vert u(t,\cdot)\Vert_{L^2} +& \Vert (-\Delta)^{\frac{s}{2}}u(t,\cdot)\Vert_{L^2} + \Vert (-\Delta)^{s}u(t,\cdot)\Vert_{L^2} + \Vert u_t(t,\cdot)\Vert_{L^2} + \Vert (-\Delta)^{\frac{s}{2}}u_t(t,\cdot)\Vert_{L^2}\big\}\nonumber\\ &\lesssim
         \Big(1+\Vert m\Vert_{L^{\frac{d}{s}}}\Big)\Big(1+\Vert m\Vert_{L^{\frac{d}{2s}}}\Big)\Big(1+\Vert b\Vert_{L^{\frac{d}{s}}}\Big)^2\bigg[\Vert u_0\Vert_{H^{2s}} + \Vert u_{1}\Vert_{H^s}\bigg],
    \end{align}
     uniformly in $t\in [0,T]$. This completes the proof.
\end{proof}

\subsection{Non-homogeneous equation case}
Here, we prove estimates to the solution of \eqref{Equation intro} in the case when $f\not\equiv 0$.

\begin{thm} \label{thm non-homogeneous}
    Let $g\in L^{\infty}(\mathbb{R}^d)$ be positive, $m,b\in L^{\infty}(\mathbb{R}^d)$ be non-negative  and $f \in C([0,T],L^{2}(\mathbb{R}^d))$. Assume that $(u_0,u_1) \in H^{s}(\mathbb{R}^d)\times L^{2}(\mathbb{R}^d)$. Then, the unique solution $u\in C([0,T];H^{s}(\mathbb{R}^d))\cap C^1([0,T];L^{2}(\mathbb{R}^d))$ to the Cauchy problem \eqref{Equation intro} satisfies the estimate
    \begin{equation}\label{Energy estimate non}
         \Vert u(t,\cdot)\Vert_1 \lesssim \big(1 + \Vert g \Vert_{L^{\infty}} + \Vert m\Vert_{L^{\infty}}\big) \big(1+ \Vert b \Vert_{L^{\infty}} +  \Vert m \Vert_{L^{\infty}}^{\frac{1}{2}}\big) \big[ \Vert u_0\Vert_{H^s} + \Vert u_1\Vert_{L^2} +\Vert f(t,\cdot)\Vert_{L^2} \big]  ,
    \end{equation}
     for all $t\in [0,T]$.
     
     Moreover, if $d>2s$, then for $g\in L^{\frac{d}{2s}}(\mathbb{R}^d)\cap L^{\frac{d}{s}}(\mathbb{R}^d)\cap W^{s,\frac{d}{s}}(\mathbb{R}^d)$, $m\in L^{\frac{d}{2s}}(\mathbb{R}^d)\cap L^{\frac{d}{s}}(\mathbb{R}^d)$, $b\in L^{\frac{d}{s}}(\mathbb{R}^d)$, $f\in C([0,T],L^{2}(\mathbb{R}^d))$, and $u_0 \in H^{3s}(\mathbb{R}^d)$, $u_1 \in H^{2s}(\mathbb{R}^d)$, there is a unique solution $u\in C\big([0,T]; H^{s}(\mathbb{R}^d)\cap C^{1}\big([0,T]; H^{s}(\mathbb{R}^d)\big)$ to \eqref{Equation intro} that satisfies the estimate
     \begin{align}\label{Energy1 aux1+1 non}
        \Vert u(t,\cdot)\Vert_3 &\lesssim \big( 1+\Vert m\Vert_{L^{\frac{d}{s}}}+\Vert b\Vert_{L^{\frac{d}{s}}} \big ) \Big[ 1+ \Vert m\Vert_{L^{\frac{d}{s}}} +\Vert b\Vert_{L^{\frac{d}{s}}} +\Vert g\Vert_{L^{\frac{d}{2s}}}^{\frac{1}{2}} +\Vert m\Vert_{L^{\frac{d}{2s}}}^{\frac{1}{2}} \\ &\quad+\Vert b\Vert_{L^{\frac{d}{s}}} \Vert g\Vert_{W^{s,\frac{d}{s}}} \Big]  \Big[ \Vert u_0\Vert_{H^{3s}} +\Vert u_1\Vert_{H^{2s}} +\Vert f(t,\cdot)\Vert_{L^2}\Big] \nonumber,
    \end{align}
   uniformly in $t\in [0,T]$.
\end{thm}

\begin{proof}
    We apply Duhamel's principle. The solution to the non-homogeneous problem  \eqref{Equation intro} is given by
    \begin{equation}\label{solution non}
u(t,x)=W(t,x)+\int_{0}^{T}V(t,x;\tau)d\tau,
\end{equation}
where $W(t,x)$ is the solution to the homogeneous problem
    \begin{equation*}\label{Equation1 non}
        \bigg\{
        \begin{array}{l}
        W_{tt}(t,x) +D_g^s W(t,x) + m(x)W(t,x)+b(x)W_t(t,x)=0,\quad (t,x)\in [0,T]\times\mathbb{R}^d, \\
        W(0,x)=u_{0}(x),\quad W_{t}(0,x)=u_{1}(x),\quad x\in \mathbb R^d,
        \end{array}
    \end{equation*}
    and $V(t,x;\tau)$ solves the problem
\begin{equation*}
    \left\lbrace
    \begin{array}{l}
    V_{tt}(t,x;\tau)+D_g^s V(t,x;\tau) +m(x)V(t,x;\tau)+b(x)V_t(t,x;\tau)=0,\\ \hfill(t,x)\in\left(\tau,T\right]\times \mathbb{R}^{d},\\
    V(\tau,x;\tau)=0, ~V_{t}(\tau,x;\tau)=f(\tau,x).
    \end{array}
    \right.
\end{equation*}
By taking $L^2$-norm in \eqref{solution non}, we get
\begin{equation}
    \Vert u(t,\cdot)\Vert_{L^2}  \leq \Vert W_{\varepsilon}(t,\cdot)\Vert_{L^2} + \int_{0}^{T}\Vert V_{\varepsilon}(t,\cdot;s)\Vert_{L^2} ds,\label{Duhamel1 solution estimate non}    
\end{equation}
  where we used  Minkowski's integral inequality. 
  We have
    \begin{equation*}
        \Vert W_{\varepsilon}(t,\cdot)\Vert_{L^2} \lesssim \big(1 + \Vert g \Vert_{L^{\infty}} + \Vert m\Vert_{L^{\infty}}\big) \big(1+ \Vert b \Vert_{L^{\infty}} +  \Vert m \Vert_{L^{\infty}}^{\frac{1}{2}}\big) \big[ \Vert u_0\Vert_{H^s} + \Vert u_1\Vert_{L^2}  \big],
    \end{equation*}
    and
    \begin{equation*}
        \Vert V_{\varepsilon}(t,\cdot;s)\Vert_{L^2} \lesssim  \big(1 + \Vert g \Vert_{L^{\infty}} + \Vert m\Vert_{L^{\infty}}\big) \big(1+ \Vert b \Vert_{L^{\infty}} +  \Vert m \Vert_{L^{\infty}}^{\frac{1}{2}}\big) \Vert f(t,\cdot)\Vert_{L^2}.
    \end{equation*} 
    The last estimate follows from \eqref{Energy estimate}. Thus, we get
    \begin{align}
        \Vert u(t,\cdot )\Vert_{L^2} 
        &\lesssim \big(1 + \Vert g \Vert_{L^{\infty}} + \Vert m\Vert_{L^{\infty}}\big) \big(1+ \Vert b \Vert_{L^{\infty}} +  \Vert m \Vert_{L^{\infty}}^{\frac{1}{2}}\big) \big[ \Vert u_0\Vert_{H^s} + \Vert u_1\Vert_{L^2}  \big] \nonumber \\ &\quad+ \big(1 + \Vert g \Vert_{L^{\infty}} + \Vert m\Vert_{L^{\infty}}\big) \big(1+ \Vert b \Vert_{L^{\infty}} +  \Vert m \Vert_{L^{\infty}}^{\frac{1}{2}}\big) \Vert f(t,\cdot)\Vert_{L^2}  \nonumber \\
        & \lesssim \big(1 + \Vert g \Vert_{L^{\infty}} + \Vert m\Vert_{L^{\infty}}\big) \big(1+ \Vert b \Vert_{L^{\infty}} +  \Vert m \Vert_{L^{\infty}}^{\frac{1}{2}}\big) \big[ \Vert u_0\Vert_{H^s} + \Vert u_1\Vert_{L^2} \nonumber \\ &\quad+\Vert f(t,\cdot)\Vert_{L^2}  \big] \nonumber .
    \end{align}
     Now, let $d>2s$. By similar arguments and by using the estimate \eqref{Energy1 aux1+1} we easily prove that:
    \begin{align}
        \Vert u(t,\cdot )\Vert_{L^2} &\lesssim  \big( 1+\Vert m\Vert_{L^{\frac{d}{s}}}+\Vert b\Vert_{L^{\frac{d}{s}}} \big ) \big [ 1+ \Vert m\Vert_{L^{\frac{d}{s}}} +\Vert b\Vert_{L^{\frac{d}{s}}} +\Vert g\Vert_{L^{\frac{d}{2s}}}^{\frac{1}{2}} +\Vert m\Vert_{L^{\frac{d}{2s}}}^{\frac{1}{2}} \nonumber \\ & \quad+\Vert b\Vert_{L^{\frac{d}{s}}} \Vert g\Vert_{W^{s,\frac{d}{s}}} \big] \big[ \Vert u_0\Vert_{H^{3s}} +\Vert u_1\Vert_{H^{2s}} \big] +\big( 1+\Vert m\Vert_{L^{\frac{d}{s}}}+\Vert b\Vert_{L^{\frac{d}{s}}} \big ) \nonumber \\ &\quad \times \big [ 1+ \Vert m\Vert_{L^{\frac{d}{s}}} +\Vert b\Vert_{L^{\frac{d}{s}}} +\Vert g\Vert_{L^{\frac{d}{2s}}}^{\frac{1}{2}} +\Vert m\Vert_{L^{\frac{d}{2s}}}^{\frac{1}{2}} +\Vert b\Vert_{L^{\frac{d}{s}}} \Vert g\Vert_{W^{s,\frac{d}{s}}} \big] \Vert f(t,\cdot)\Vert_{L^{2}}  \nonumber \\ &\lesssim \big( 1+\Vert m\Vert_{L^{\frac{d}{s}}}+\Vert b\Vert_{L^{\frac{d}{s}}} \big ) \big [ 1+ \Vert m\Vert_{L^{\frac{d}{s}}} +\Vert b\Vert_{L^{\frac{d}{s}}} +\Vert g\Vert_{L^{\frac{d}{2s}}}^{\frac{1}{2}} +\Vert m\Vert_{L^{\frac{d}{2s}}}^{\frac{1}{2}} \nonumber \\ &\quad +\Vert b\Vert_{L^{\frac{d}{s}}} \Vert g\Vert_{W^{s,\frac{d}{s}}} \big] \big[ \Vert u_0\Vert_{H^{3s}} +\Vert u_1\Vert_{H^{2s}} +\Vert f(t,\cdot)\Vert_{L^2} \big], \nonumber
    \end{align}
    ending the proof.    
\end{proof}
In the following theorem, we prove estimates to the solution of \eqref{Equation intro} in the special case when $f\not\equiv 0$ and $g\equiv 1$.

\begin{thm} \label{thm non-homogeneous 1}
    Let $g\equiv 1$, $m,b\in L^{\infty}(\mathbb{R}^d)$ be non-negative  and $f \in C([0,T],L^{2}(\mathbb{R}^d))$. Suppose that $(u_0,u_1) \in H^{s}(\mathbb{R}^d)\times L^{2}(\mathbb{R}^d)$. Then, the unique solution\\ $u\in C([0,T];H^{s}(\mathbb{R}^d))\cap C^1([0,T];L^{2}(\mathbb{R}^d))$ to the Cauchy problem \eqref{Equation intro} satisfies the estimate
    \begin{equation}\label{Energy estimate non 1}
         \Vert u(t,\cdot)\Vert_1 \lesssim \big(2 + \Vert m\Vert_{L^{\infty}}\big) \big(1+ \Vert b \Vert_{L^{\infty}} +  \Vert m \Vert_{L^{\infty}}^{\frac{1}{2}}\big) \big[ \Vert u_0\Vert_{H^s} + \Vert u_1\Vert_{L^2} +\Vert f(t,\cdot)\Vert_{L^2} \big]  ,
    \end{equation}
     for all $t\in [0,T]$.
     
      Moreover, if $d>2s$, then for $g\equiv 1 $, $m\in L^{\frac{d}{2s}}(\mathbb{R}^d)\cap L^{\frac{d}{s}}(\mathbb{R}^d)$, $b\in L^{\frac{d}{s}}(\mathbb{R}^d)$, $f\in C([0,T],L^{2}(\mathbb{R}^d))$, and $u_0 \in H^{2s}(\mathbb{R}^d)$, $u_1 \in H^{s}(\mathbb{R}^d)$, there is a unique solution $u\in C\big([0,T]; H^{2s}(\mathbb{R}^d)\cap C^{1}\big([0,T]; H^{s}(\mathbb{R}^d)\big)$ to \eqref{Equation intro} that satisfies the estimate
     \begin{align}\label{Energy1 aux1+1 non 1}
        \Vert u(t,\cdot)\Vert_2 \lesssim \big(1+\Vert m\Vert_{L^{\frac{d}{s}}}\big)\big(1+\Vert m\Vert_{L^{\frac{d}{2s}}}\big)&\big(1+\Vert b\Vert_{L^{\frac{d}{s}}}\big)^2 \big[\Vert u_0\Vert_{H^{2s}} \\ &\quad+ \Vert u_{1}\Vert_{H^s}  +\Vert f(t,\cdot)\Vert_{L^2}\big],\nonumber
    \end{align}
   uniformly in $t\in [0,T]$.
\end{thm}
\begin{proof}
    The first estimate follows from \eqref{Energy estimate non}, by taking $\Vert g\Vert_{L^{\infty}} =1$. We get 
     \begin{equation*}
     \Vert u(t,\cdot)\Vert_1 \lesssim \big(2 + \Vert m\Vert_{L^{\infty}}\big) \big(1+ \Vert b \Vert_{L^{\infty}} +  \Vert m \Vert_{L^{\infty}}^{\frac{1}{2}}\big) \big[ \Vert u_0\Vert_{H^s} + \Vert u_1\Vert_{L^2} +\Vert f(t,\cdot)\Vert_{L^2} \big].
     \end{equation*}
    Now, let $d>2s$. By applying Duhamel's principle, the solution to the Non-homogeneous problem  \eqref{Equation intro}, when $g\equiv 1$ is given by
    \begin{equation}\label{solution non g1}
u(t,x)=W(t,x)+\int_{0}^{T}V(t,x;\tau)d\tau, 
\end{equation}
where $W(t,x)$ is the solution to the homogeneous problem
    \begin{equation*}\label{Equation1 non g1}
        \bigg\{
        \begin{array}{l}
        W_{tt}(t,x) +(-\Delta)^s W(t,x) + m(x)W(t,x)+b(x)W_t(t,x)=0,\quad (t,x)\in [0,T]\times\mathbb{R}^d, \\
        W(0,x)=u_{0}(x),\quad W_{t}(0,x)=u_{1}(x),\quad x\in \mathbb R^d,
        \end{array}
    \end{equation*}
    and $V(t,x;\tau)$ solves the problem
\begin{equation*}
    \left\lbrace
    \begin{array}{l}
    V_{tt}(t,x;\tau)+(-\Delta)^s V(t,x;\tau) +m(x)V(t,x;\tau)+b(x)V_t(t,x;\tau)=0,\\ \hfill(t,x)\in\left(\tau,T\right]\times \mathbb{R}^{d},\\
    V(\tau,x;\tau)=0, ~V_{t}(\tau,x;\tau)=f(\tau,x).
    \end{array}
    \right.
\end{equation*}
By taking $L^2$-norm in \eqref{solution non g1} and by using the estimate \eqref{Energy estimate1}, we get
\begin{equation}
    \Vert u_(t,\cdot)\Vert_{L^2}  \leq \Vert W_{\varepsilon}(t,\cdot)\Vert_{L^2} + \int_{0}^{T}\Vert V_{\varepsilon}(t,\cdot;s)\Vert_{L^2} ds,\label{Duhamel1 solution estimate non g1}    
\end{equation}
  where we used  Minkowski's integral inequality. We have
    \begin{equation*}
        \Vert W_{\varepsilon}(t,\cdot)\Vert_{L^2} \lesssim \big(1+\Vert m\Vert_{L^{\frac{d}{s}}}\big)\big(1+\Vert m\Vert_{L^{\frac{d}{2s}}}\big)\big(1+\Vert b\Vert_{L^{\frac{d}{s}}}\big)^2\big[\Vert u_0\Vert_{H^{2s}} + \Vert u_{1}\Vert_{H^s}\big],
    \end{equation*}
    and
    \begin{equation*}
        \Vert V_{\varepsilon}(t,\cdot;s)\Vert_{L^2} \lesssim   \big(1+\Vert m\Vert_{L^{\frac{d}{s}}}\big)\big(1+\Vert m\Vert_{L^{\frac{d}{2s}}}\big)\big(1+\Vert b\Vert_{L^{\frac{d}{s}}}\big)^2 \Vert f(t,\cdot)\Vert_{L^2}.
    \end{equation*}
    Thus
    \begin{align}
        \Vert u(t,\cdot )\Vert_{L^2} 
        &\lesssim\big(1+\Vert m\Vert_{L^{\frac{d}{s}}}\big)\big(1+\Vert m\Vert_{L^{\frac{d}{2s}}}\big)\big(1+\Vert b\Vert_{L^{\frac{d}{s}}}\big)^2\big[\Vert u_0\Vert_{H^{2s}} + \Vert u_{1}\Vert_{H^s}\big] \nonumber \\ &\quad+ \big(1+\Vert m\Vert_{L^{\frac{d}{s}}}\big)\big(1+\Vert m\Vert_{L^{\frac{d}{2s}}}\big)\big(1+\Vert b\Vert_{L^{\frac{d}{s}}}\big)^2 \Vert f(t,\cdot)\Vert_{L^2}  \nonumber \\
        & \lesssim \big(1+\Vert m\Vert_{L^{\frac{d}{s}}}\big)\big(1+\Vert m\Vert_{L^{\frac{d}{2s}}}\big)\big(1+\Vert b\Vert_{L^{\frac{d}{s}}}\big)^2\big[\Vert u_0\Vert_{H^{2s}} + \Vert u_{1}\Vert_{H^s} \nonumber \\ &\quad+\Vert f(t,\cdot)\Vert_{L^2}  \big] \nonumber .
    \end{align} 
    This completes the proof.
\end{proof}
\section{Very weak well-posedness}
Here and in the sequel, we consider the case when the equation coefficients $g$, $m$, $b$ and $f$, and the Cauchy data $u_0$ and $u_1$ are irregular functions and prove that the Cauchy problem
\begin{equation}\label{Equation sing}
    \bigg\lbrace
    \begin{array}{l}
    u_{tt}(t,x) + D_g^{s}u(t,x) + m(x)u(t,x) + b(x)u_{t}(t,x)=f(t,x), \,\, (t,x)\in [0,T]\times\mathbb{R}^d,\\
    u(0,x)=u_{0}(x),\quad u_{t}(0,x)=u_{1}(x), \quad x\in\mathbb{R}^d,
    \end{array}
\end{equation}
has a unique very weak solution. We have in mind ``functions" having $\delta$ or $\delta^2$-like behaviours.

\subsection{Existence of very weak solutions}
We are ready now to introduce the notion of very weak solutions adapted to our problem. We start by regularising the equation coefficients $g,m,b,f$ and the Cauchy data $u_0$ and $u_1$ by convolution with a mollifying net $\left(\psi_\varepsilon\right)_{\varepsilon\in(0,1]}$, generating families of smooth functions $(g_{\varepsilon})_{\varepsilon}$, $(m_{\varepsilon})_{\varepsilon}$, $(b_{\varepsilon})_{\varepsilon}$, $(f_{\varepsilon})_{\varepsilon}$, and $(u_{0,\varepsilon})_{\varepsilon}$, $(u_{1,\varepsilon})_{\varepsilon}$. Namely,
$$
g_{\varepsilon}(x)=g\ast \psi_{\varepsilon }(x),~~ g_{\varepsilon}(x)=g\ast \psi_{\varepsilon }(x),~~ b_{\varepsilon}(x)=b\ast \psi_{\varepsilon }(x),~~ f_{\varepsilon}(t,x)=f\ast \psi_{\varepsilon }(x),$$
and
$$ u_{0,\varepsilon}(x)=u_0\ast \psi_{\varepsilon }(x),~~
u_{1,\varepsilon}(x)=u_1\ast \psi_{\varepsilon }(x), $$
where
$
\psi_{\varepsilon}(x)=\varepsilon^{-d}\psi(\varepsilon^{-1}x)$ and $\varepsilon\in\left(0,1\right]$.
Now, instead of considering \eqref{Equation sing}, we consider the family of regularised problems
\begin{equation}\label{Regularised equation0}
        \left\lbrace
        \begin{array}{l}
        \partial_{t}^2u_{\varepsilon}(t,x) + D_g^{s}u_{\varepsilon}(t,x) + m_{\varepsilon}(x)u_{\varepsilon}(t,x) + b_{\varepsilon}(x)\partial_{t}u_{\varepsilon}(t,x)=f_{\varepsilon}(t,x),\\ \hfill (t,x)\in [0,T]\times\mathbb{R}^d,\\
        u_{\varepsilon}(0,x)=u_{0,\varepsilon}(x),\quad \partial_{t}u_{\varepsilon}(0,x)=u_{1,\varepsilon}(x), \quad x\in\mathbb{R}^d.
        \end{array}
        \right.
    \end{equation}

\begin{defn}[\textbf{Moderateness}]~
\begin{enumerate}
    \item[1). ] For $T>0$. A net of functions $(u_\varepsilon(\cdot,\cdot))_{\varepsilon\in(0,1]}$ from $C\left([0,T],H^s (\mathbb R^d)\right)\cap$\\ $C^1\left([0,T],L^2 (\mathbb R^d)\right)$ is said to be $C\left([0,T],H^s (\mathbb R^d)\right)\cap C^1\left([0,T],L^2 (\mathbb R^d)\right)$-moderate, if there exist $N\in\mathbb{N}^\ast$ such that
    \begin{equation}
        \sup_{t\in [0,T]}\Vert u_\varepsilon(t,\cdot)\Vert_1 \lesssim \varepsilon^{-N}.
    \end{equation}
    \item[2).] For $T>0$. A net of functions $(u_\varepsilon(\cdot,\cdot))_{\varepsilon\in(0,1]}$ from $C\left([0,T],H^{2s} (\mathbb R^d)\right)\cap$\\ $C^1\left([0,T],H^s (\mathbb R^d)\right)$ is said to be $C\left([0,T],H^{2s} (\mathbb R^d)\right)\cap C^1\left([0,T],H^s (\mathbb R^d)\right)$-moderate, if there exist $N\in\mathbb{N}^\ast$ such that
    \begin{equation}
        \sup_{t\in [0,T]}\Vert u_\varepsilon(t,\cdot)\Vert_2 \lesssim \varepsilon^{-N}.
    \end{equation}
    \item[3).] For $T>0$. A net of functions $(u_\varepsilon(\cdot,\cdot))_{\varepsilon\in(0,1]}$ from $C\left([0,T],H^s (\mathbb R^d)\right)\cap$\\ $C^1\left([0,T],H^s (\mathbb R^d)\right)$ is said to be $C\left([0,T],H^s (\mathbb R^d)\right)\cap C^1\left([0,T],H^s (\mathbb R^d)\right)$-moderate, if there exist $N\in\mathbb{N}^\ast$ such that
    \begin{equation}
        \sup_{t\in [0,T]}\Vert u_\varepsilon(t,\cdot)\Vert_3 \lesssim \varepsilon^{-N}.
    \end{equation}
\end{enumerate}
We will shortly write $C_1$-moderate, $C_2$-moderate, $C_3$-moderate instead of\\ $C\left([0,T],H^s (\mathbb R^d)\right)\cap C^1\left([0,T],L^2 (\mathbb R^d)\right)$-moderate, $C\left([0,T],H^{2s} (\mathbb R^d)\right)\cap $\\ $ C^1\left([0,T],H^s (\mathbb R^d)\right)$-moderate, $C\left([0,T],H^s (\mathbb R^d)\right)\cap  C^1\left([0,T],H^s (\mathbb R^d)\right)$-moderate, respectively.
\end{defn}
In the following proposition, we prove that distributional coefficients are naturally moderate.

\begin{prop}\label{coherence 1}~
    \begin{itemize}
        \item[(1)]  Let $f\in \mathcal{E}^{'}(\mathbb{R}^d)$ and let $(f_\varepsilon)_{\varepsilon\in(0,1]}$ be regularisation of $f$ obtained via convolution with a mollifying net $\left(\psi_\varepsilon\right)_{\varepsilon\in(0,1]}$. Then, the net $(f_\varepsilon)_{\varepsilon\in(0,1]}$ is $L^{p}(\mathbb{R}^d)$-moderate for any $1\leq p\leq \infty$. That is, there exist $N\in \mathbb{N^*}$ such that 
         $$\Vert f_{\varepsilon}\Vert_{L^p} \leq \varepsilon^{-N}.$$
       \item[(2)] Let $f\in L^p(\mathbb{R}^d)$. Then the net $(f_\varepsilon)_{\varepsilon\in(0,1]}$ is $L^{p}(\mathbb{R}^d)$-moderate for any $1\leq p\leq \infty$.
    \end{itemize}
\end{prop}
\begin{proof}
    (1)  Fix $p\in [1,\infty]$ and let $f\in \mathcal{E}^{'}(\mathbb{R}^d)$. By the structure theorems for distributions, there exists $n\in \mathbb{N}$ and compactly supported functions $f_{\alpha}\in C(\mathbb{R}^{d})$ such that
    \begin{equation*}
    T=\sum_{|\alpha| \leq n}\partial^{\alpha}f_{\alpha},
    \end{equation*}
    where $|\alpha|$ is the length of the multi-index $\alpha$. The convolution of $f$ with a mollifying net $\left(\psi_\varepsilon\right)_{\varepsilon\in(0,1]}$ yields
    \begin{equation}\label{Struct thm}
        f\ast\psi_{\varepsilon}=\sum_{\vert \alpha\vert \leq n}\partial^{\alpha}f_{\alpha}\ast\psi_{\varepsilon}=\sum_{\vert \alpha\vert \leq n}f_{\alpha}\ast\partial^{\alpha}\psi_{\varepsilon}=\sum_{\vert \alpha\vert \leq n}\varepsilon^{-d-\vert\alpha\vert}f_{\alpha}\ast\partial^{\alpha}\psi(x/\varepsilon).
    \end{equation}
    Taking the $L^p$ norm in \eqref{Struct thm} gives
    \begin{equation}\label{Struct thm1}
        \Vert f\ast\psi_{\varepsilon}\Vert_{L^p} \leq \sum_{\vert \alpha\vert \leq n}\varepsilon^{-d-\vert\alpha\vert}\Vert f_{\alpha}\ast\partial^{\alpha}\psi(x/\varepsilon)\Vert_{L^p}.
    \end{equation}
    Since $f_{\alpha}$ and $\psi$ are compactly supported then, Young's inequality applies for any $p_1,p_2 \in [1,\infty]$, provided that $\frac{1}{p_1}+\frac{1}{p_2}=\frac{1}{p}+1$. That is
    \begin{equation*}
        \Vert f_{\alpha}\ast\partial^{\alpha}\psi(x/\varepsilon)\Vert_{L^p} \leq \Vert f_{\alpha}\Vert_{L^{p_1}}\Vert\partial^{\alpha}\psi(x/\varepsilon)\Vert_{L^{p_2}} <\infty.
    \end{equation*}
    It follows from \eqref{Struct thm1} that $(f_\varepsilon)_{\varepsilon\in(0,1]}$ is $L^{p}(\mathbb{R}^d)$-moderate. This completes the proof of the first statement.\\
    (2)  In order to prove the second statement of the proposition, we use Young's convolution inequality. Let $1\leq p\leq \infty$ and $f\in L^p(\mathbb{R}^d)$. We have:
    \begin{equation*}
        \Vert f_{\varepsilon} \Vert_{L^p} =\Vert f\ast \psi_{\varepsilon} \Vert_{L^p} \leq \Vert f\Vert_{L^{p_1}} \Vert \psi_{\varepsilon} \Vert_{L^{p_2}},
    \end{equation*}
    such that $\frac{1}{p_1}+\frac{1}{p_2}=\frac{1}{p}+1$. If we choose $p_1=p$ and consequently $p_2=1$, we get
    \begin{align*}
        \Vert f_{\varepsilon} \Vert_{L^p}&\leq \Vert f\Vert_{L^{p}} \Vert \psi_{\varepsilon} \Vert_{L^{1}} \\
        &\lesssim \Vert f\Vert_{L^p} \varepsilon^{-d} \\
        &\lesssim \varepsilon^{-d},
    \end{align*}
    where we used that
    \begin{align*}
        \Vert \psi_{\varepsilon} \Vert_{L^{1}}=\int_{\mathbb{R}^d} \psi_{\varepsilon}(x)dx&=\int_{\mathbb{R}^d} \varepsilon^{-1} \psi (\varepsilon^{-1}x)dx \\ &=\varepsilon^{-1}\int_{\mathbb{R}^d} \psi(\varepsilon^{-1}x)dx \\
        &\lesssim \varepsilon^{-1}.
    \end{align*}
    The integral $\int_{\mathbb{R}^d} \psi(\varepsilon^{-1}x)dx<\infty$, since $\psi$ is compactly supported.
\end{proof}

\subsubsection{Existence of a very weak solution: general case}

\begin{defn}[\textbf{Very weak solution: case 1}]\label{Defn1 V.W.S}
    A net of functions $(u_{\varepsilon})_{\varepsilon\in (0,1]}\in C([0,T];H^{s}(\mathbb{R}^d))$ $\cap C^1([0,T];L^{2}(\mathbb{R}^d))$ is said to be a very weak solution to the Cauchy problem (\ref{Equation sing}), if there exist
    \begin{itemize}
        \item $L^{\infty}(\mathbb{R}^d)$-moderate regularisations $(g_{\varepsilon})_{\varepsilon}$ and $(m_{\varepsilon})_{\varepsilon}$ and $(b_{\varepsilon})_{\varepsilon}$ to $g$ and $m$ and $b$, with $g_{\varepsilon}>0$ and $m_{\varepsilon} \geq 0$ and $b_{\varepsilon} \geq 0$,
        \item $L^{2}(\mathbb{R}^d)$-moderate regularisation $(f_{\varepsilon}(t,\cdot))_{\varepsilon}$ to $f(t,\cdot)$, for any $t\in [0,T]$,
        \item $H^{s}(\mathbb{R}^d)$-moderate regularisation $(u_{0,\varepsilon})_{\varepsilon}$ to $u_0$,
        \item $L^{2}(\mathbb{R}^d)$-moderate regularisation $(u_{1,\varepsilon})_{\varepsilon}$ to $u_1$,
    \end{itemize}
    such that $(u_{\varepsilon})_{\varepsilon\in (0,1]}$ solves the regularised problems
    \begin{equation}\label{Regularised equation}
        \left\lbrace
        \begin{array}{l}
        \partial_{t}^2u_{\varepsilon}(t,x) + D_{g_{\varepsilon}}^{s}u_{\varepsilon}(t,x) + m_{\varepsilon}(x)u_{\varepsilon}(t,x) + b_{\varepsilon}(x)\partial_{t}u_{\varepsilon}(t,x)=f_{\varepsilon}(t,x),\\ \hfill (t,x)\in [0,T]\times\mathbb{R}^d,\\
        u_{\varepsilon}(0,x)=u_{0,\varepsilon}(x),\quad \partial_{t}u_{\varepsilon}(0,x)=u_{1,\varepsilon}(x), \quad x\in\mathbb{R}^d,
        \end{array}
        \right.
    \end{equation}
    for all $\varepsilon\in (0,1]$, and is $C_1$-moderate.
\end{defn}

We have also the following alternative definition of a very weak solution to \eqref{Equation sing}, under the assumptions of Theorem \ref{thm non-homogeneous}, when $d>2s$.

\begin{defn}[\textbf{Very weak solution: case 2}]\label{Defn2 V.W.S}
    Let $d>2s$. A net of functions $(u_{\varepsilon})_{\varepsilon}\in C([0,T];H^{s}(\mathbb{R}^d))\cap C^1([0,T];H^{s}(\mathbb{R}^d))$ is said to be a very weak solution to the Cauchy problem (\ref{Equation sing}), if there exist
    \begin{itemize}
        \item  $\big(L^{\frac{d}{s}}(\mathbb{R}^d)\cap L^{\frac{d}{2s}}(\mathbb{R}^d)\cap W^{s,\frac{d}{s}}(\mathbb {R}^d)\big)$-moderate regularisation $(g_{\varepsilon})_{\varepsilon}$ to $g$,\\
        with $g_{\varepsilon} > 0$,
        \item $\big(L^{\frac{d}{s}}(\mathbb{R}^d)\cap L^{\frac{d}{2s}}(\mathbb{R}^d)\big)$-moderate regularisation $(m_{\varepsilon})_{\varepsilon}$ to $m$, with $m_{\varepsilon} \geq 0$,
        \item $L^{\frac{d}{s}}(\mathbb{R}^d)$-moderate regularisation $(b_{\varepsilon})_{\varepsilon}$ to $b$, with $b_{\varepsilon} \geq 0$,
        \item $L^{2}(\mathbb{R}^d)$-moderate regularisation $(f_{\varepsilon}(t,\cdot))_{\varepsilon}$ to $f(t,\cdot)$, for any $t\in [0,T]$, 
        \item $H^{3s}(\mathbb{R}^d)$-moderate regularisation $(u_{0,\varepsilon})_{\varepsilon}$ to $u_0$,
        \item $H^{2s}(\mathbb{R}^d)$-moderate regularisation $(u_{1,\varepsilon})_{\varepsilon}$ to $u_1$,
    \end{itemize}
    such that $(u_{\varepsilon})_{\varepsilon}$ solves the regularised problems, and is $C_3$-moderate.
\end{defn}
Now, under the assumptions in Definition \ref{Defn1 V.W.S} and Definition \ref{Defn2 V.W.S}, the existence of a very weak solution is straightforward.

\begin{thm}\label{Thm1 existence}
    Assume that there exist\\
    $\big\{L^{\infty}(\mathbb{R}^d),L^{\infty}(\mathbb{R}^d),L^{\infty}(\mathbb{R}^d),L^2(\mathbb{R}^d),H^{s}(\mathbb{R}^d),L^2(\mathbb{R}^d)\big\}$-moderate regularisations to $g,m,b,f$ and $u_0,u_1$ respectively, with $g_{\varepsilon} > 0$ and $m_{\varepsilon} \geq 0$ and $b_{\varepsilon} \geq 0$. Then, the Cauchy problem \eqref{Equation sing} has a very weak solution.
\end{thm}

\begin{proof}
    Let $g,m,b,f$ and $u_0,u_1$ as in assumptions. Then, there exists\\ $N_1,N_2,N_3,N_4,N_5,N_6 \in \mathbb{N}^\ast$, such that
    \begin{equation*}
        \Vert g_{\varepsilon}\Vert_{L^{\infty}} \lesssim \varepsilon^{-N_1},\quad \Vert m_{\varepsilon}\Vert_{L^{\infty}} \lesssim \varepsilon^{-N_2},\quad \Vert b_{\varepsilon}\Vert_{L^{\infty}} \lesssim \varepsilon^{-N_3},\quad \Vert f_{\varepsilon}(t,\cdot)\Vert_{L^2} \lesssim \varepsilon^{-N_4},
    \end{equation*}
    and 
    \begin{equation*}
        \Vert u_{0,\varepsilon}\Vert_{H^s} \lesssim \varepsilon^{-N_5},\quad \Vert u_{1,\varepsilon}\Vert_{L^2} \lesssim \varepsilon^{-N_6}.
    \end{equation*}
    It follows from the energy estimate \eqref{Energy estimate non}, that
    \begin{equation*}
        \Vert u_{\varepsilon}(t,\cdot)\Vert_1 \lesssim \varepsilon^{-N_1-N_2-N_3-\max \{N_4,N_5,N_6\}},
    \end{equation*}
    uniformly in $t\in [0,T]$, which means that the net $(u_{\varepsilon})_{\varepsilon}$ is $C_1$-moderate. This concludes the proof.
\end{proof}

As an alternative to Theorem \ref{Thm1 existence} in the case when $d>2s$ and the equation coefficients and data satisfy the hypothesis of Definition \ref{Defn2 V.W.S}, we have the following theorem.

\begin{thm}\label{Thm2 existence}
    Assume that there exist\\
    $\big\{\big(L^{\frac{d}{s}}(\mathbb{R}^d)\cap L^{\frac{d}{2s}}(\mathbb{R}^d)\cap W^{s,\frac{d}{s}}(\mathbb {R}^d)\big),\big(L^{\frac{d}{s}}(\mathbb{R}^d)\cap L^{\frac{d}{2s}}(\mathbb{R}^d)\big),L^{\frac{d}{s}}(\mathbb{R}^d),L^2(\mathbb{R}^d),H^{3s}(\mathbb{R}^d),\\ H^{2s}(\mathbb{R}^d)
    \big\}$-moderate regularisations to $g,m,b,f$ and $u_0,u_1$ respectively, with $g_{\varepsilon} > 0$, $m_{\varepsilon} \geq 0$ and $b_{\varepsilon} \geq 0$. Then, the Cauchy problem \eqref{Equation sing} has a very weak solution.
\end{thm}
\begin{proof}
    Let $g,m,b,f$ and $u_0,u_1$ as in assumptions. Then, there exist \\ $N_1,N_2,N_3,N_4,N_5,N_6 \in \mathbb{N}^\ast$, such that
    \begin{equation*}
        \max \big\{\Vert g_{\varepsilon}\Vert_{L^{\frac{d}{s}}}, \Vert g_{\varepsilon}\Vert_{L^{\frac{d}{2s}}}, \Vert g_{\varepsilon}\Vert_{W^{s,\frac{d}{s}}}\big\} \lesssim \varepsilon^{-N_1}, \quad \max \big\{\Vert m_{\varepsilon}\Vert_{L^{\frac{d}{s}}}, \Vert m_{\varepsilon}\Vert_{L^{\frac{d}{2s}}}\big\} \lesssim \varepsilon^{-N_2},
    \end{equation*}
    \begin{equation*}
        \Vert b_{\varepsilon}\Vert_{L^{\frac{d}{s}}} \lesssim \varepsilon^{-N_3},\quad \Vert f_{\varepsilon}\Vert_{L^2} \lesssim \varepsilon^{-N_4}
    \end{equation*}
    and 
    \begin{equation*}
        \Vert u_{0,\varepsilon}\Vert_{H^{3s}} \lesssim \varepsilon^{-N_5},\quad \Vert u_{1,\varepsilon}\Vert_{H^{2s}} \lesssim \varepsilon^{-N_6}.
    \end{equation*}
    It follows from the energy estimate \eqref{Energy1 aux1+1 non}, that
    \begin{equation*}
        \Vert u_{\varepsilon}(t,\cdot)\Vert_3 \lesssim \varepsilon^{-\max \{N_2,N_3\}-\max \{N_2,N_1 +N_3\}-\max \{N_4,N_5,N_6\}},
    \end{equation*}
    uniformly in $t\in [0,T]$, which means that the net $(u_{\varepsilon})_{\varepsilon}$ is $C_3$-moderate. This concludes the proof.
\end{proof}

\subsubsection{Existence of a very weak solution: case $g\equiv 1$}

The definitions of very weak solutions adapted to our problem in the case when $g\equiv 1$, that is for the problem
\begin{equation}\label{Equation sing g1}
    \left\lbrace
        \begin{array}{l}
    u_{tt}(t,x) + (-\Delta)^{s}u(t,x) + m(x)u(t,x) + b(x)u_{t}(t,x)=f(t,x), \\ \hfill(t,x)\in [0,T]\times\mathbb{R}^d,\\
    u(0,x)=u_{0}(x),\quad u_{t}(0,x)=u_{1}(x), \quad x\in\mathbb{R}^d,
    \end{array}
    \right.
\end{equation}
read as follows:
\begin{defn}[\textbf{Very weak solution: case 1}]\label{Defn1 V.W.S g1}
    Let $g\equiv 1$. A net of functions $(u_{\varepsilon})_{\varepsilon\in (0,1]}\in C([0,T];H^{s}(\mathbb{R}^d))$ $\cap C^1([0,T];L^{2}(\mathbb{R}^d))$ is said to be a very weak solution to the Cauchy problem (\ref{Equation sing g1}), if there exist
    \begin{itemize}
        \item $L^{\infty}(\mathbb{R}^d)$-moderate regularisations  $(m_{\varepsilon})_{\varepsilon}$ and $(b_{\varepsilon})_{\varepsilon}$ to  $m$ and $b$, with  $m_{\varepsilon} \geq 0$ and $b_{\varepsilon} \geq 0$,
        \item $L^{2}(\mathbb{R}^d)$-moderate regularisation $(f_{\varepsilon}(t,\cdot))_{\varepsilon}$ to $f(t,\cdot)$, for any $t\in [0,T]$,
        \item $H^{s}(\mathbb{R}^d)$-moderate regularisation $(u_{0,\varepsilon})_{\varepsilon}$ to $u_0$,
        \item $L^{2}(\mathbb{R}^d)$-moderate regularisation $(u_{1,\varepsilon})_{\varepsilon}$ to $u_1$,
    \end{itemize}
    such that $(u_{\varepsilon})_{\varepsilon\in (0,1]}$ solves the regularised problems
    \begin{equation}\label{Regularised equation g1}
        \left\lbrace
        \begin{array}{l}
        \partial_{t}^2u_{\varepsilon}(t,x) + (-\Delta)^{s}u_{\varepsilon}(t,x) + m_{\varepsilon}(x)u_{\varepsilon}(t,x) + b_{\varepsilon}(x)\partial_{t}u_{\varepsilon}(t,x)=f_{\varepsilon}(t,x),\\ \hfill (t,x)\in [0,T]\times\mathbb{R}^d,\\
        u_{\varepsilon}(0,x)=u_{0,\varepsilon}(x),\quad \partial_{t}u_{\varepsilon}(0,x)=u_{1,\varepsilon}(x), \quad x\in\mathbb{R}^d,
        \end{array}
        \right.
    \end{equation}
    for all $\varepsilon\in (0,1]$, and is $C_1$-moderate.
\end{defn}

In the case when $d>2s$, we have the following definition.

\begin{defn}[\textbf{Very weak solution: case 2}]\label{Defn2 V.W.S g1}
    Let $d>2s$ and $g\equiv 1$. A net of functions $(u_{\varepsilon})_{\varepsilon}\in C([0,T];H^{2s}(\mathbb{R}^d))\cap C^1([0,T];H^{s}(\mathbb{R}^d))$ is said to be a very weak solution to the Cauchy problem (\ref{Equation sing g1}), if there exist
    \begin{itemize}
        \item $\big(L^{\frac{d}{s}}(\mathbb{R}^d)\cap L^{\frac{d}{2s}}(\mathbb{R}^d)\big)$-moderate regularisation $(m_{\varepsilon})_{\varepsilon}$ to $m$, with $m_{\varepsilon} \geq 0$,
        \item $L^{\frac{d}{s}}(\mathbb{R}^d)$-moderate regularisation $(b_{\varepsilon})_{\varepsilon}$ to $b$, with $b_{\varepsilon} \geq 0$,
        \item $L^{2}(\mathbb{R}^d)$-moderate regularisation $(f_{\varepsilon}(t,\cdot))_{\varepsilon}$ to $f(t,\cdot)$, for any $t\in [0,T]$, 
        \item $H^{2s}(\mathbb{R}^d)$-moderate regularisation $(u_{0,\varepsilon})_{\varepsilon}$ to $u_0$,
        \item $H^{s}(\mathbb{R}^d)$-moderate regularisation $(u_{1,\varepsilon})_{\varepsilon}$ to $u_1$,
    \end{itemize}
    such that $(u_{\varepsilon})_{\varepsilon}$ solves the regularised problems (as in Definition \ref{Defn1 V.W.S g1}) \\ for all $\varepsilon\in (0,1]$, and is $C_2$-moderate.
\end{defn}
In what follows we want to prove the existence of a very weak solution to the Cauchy problem \eqref{Equation sing g1} under the assumptions of Definition \ref{Defn1 V.W.S g1} and Definition \ref{Defn2 V.W.S g1}.

\begin{thm}\label{Thm1 existence1}
    Let $g\equiv 1$. Assume that there exist\\
    $\big\{L^{\infty}(\mathbb{R}^d),L^{\infty}(\mathbb{R}^d),L^2(\mathbb{R}^d),H^{s}(\mathbb{R}^d),L^2(\mathbb{R}^d)\big\}$-moderate regularisations of $m,b,f$ and $u_0,u_1$, respectively, with $m_{\varepsilon} \geq 0$ and $b_{\varepsilon} \geq 0$. Then, the Cauchy problem \eqref{Equation sing g1} has a very weak solution.
\end{thm}

\begin{proof}
    Let $m,b,f$ and $u_0,u_1$ as in assumptions. Then, there exist \\ $N_1,N_2,N_3,N_4,N_5 \in \mathbb{N}^{\ast}$, such that
    \begin{equation*}
        \Vert m_{\varepsilon}\Vert_{L^{\infty}} \lesssim \varepsilon^{-N_1},\quad \Vert b_{\varepsilon}\Vert_{L^{\infty}} \lesssim \varepsilon^{-N_2},\quad \Vert f_{\varepsilon}\Vert_{L^{\infty}} \lesssim \varepsilon^{-N_3},
    \end{equation*}
    and 
    \begin{equation*}
        \Vert u_{0,\varepsilon}\Vert_{H^s} \lesssim \varepsilon^{-N_4},\quad \Vert u_{1,\varepsilon}\Vert_{L^2} \lesssim \varepsilon^{-N_5}.
    \end{equation*}
    It follows from the energy estimate \eqref{Energy estimate non 1}, that
    \begin{equation*}
        \Vert u_{\varepsilon}(t,\cdot)\Vert_1 \lesssim \varepsilon^{-N_1-\max \{\frac{N_1}{2},N_2\}-\max \{N_3,N_4,N_5\}},
    \end{equation*}
    uniformly in $t\in [0,T]$. The net $(u_{\varepsilon})_{\varepsilon}$ is then $C_1$-moderate, and the existence of a very weak solution follows, ending the proof.
\end{proof}

In the case when $d>2s$, and the equation coefficients and data satisfy the hypotheses of Definition \ref{Defn2 V.W.S g1}, we have the following theorem.

\begin{thm}\label{Thm2 existence1}
    Let $g\equiv 1$. Assume that there exist\\
    $\big\{\big(L^{\frac{d}{s}}(\mathbb{R}^d)\cap L^{\frac{d}{2s}}(\mathbb{R}^d)\big),L^{\frac{d}{s}}(\mathbb{R}^d),L^2(\mathbb{R}^d),H^{2s}(\mathbb{R}^d), H^{s}(\mathbb{R}^d)
    \big\}$-moderate regularisations to $m,b,f$ and $u_0,u_1$, respectively, with $m_{\varepsilon} \geq 0$ and $b_{\varepsilon} \geq 0$. Then, the Cauchy problem \eqref{Equation sing g1} has a very weak solution.
\end{thm}

\begin{proof}
    Let $m,b,f$ and $u_0,u_1$ as in assumptions. Then, there exist \\ $N_1,N_2,N_3,N_4,N_5 \in \mathbb{N}$, such that
    \begin{align}
         \max \big\{\Vert m_{\varepsilon}\Vert_{L^{\frac{d}{s}}} , \Vert m_{\varepsilon}\Vert_{L^{\frac{d}{2s}}}\big\} \lesssim \varepsilon^{-N_1},\quad \Vert b_{\varepsilon}\Vert_{L^{\frac{d}{s}}} \lesssim \varepsilon^{-N_2},,\quad \Vert f_{\varepsilon}\Vert_{L^2} \lesssim \varepsilon^{-N_3}\nonumber
    \end{align}
    and 
    \begin{equation*}
        \Vert u_{0,\varepsilon}\Vert_{H^{2s}} \lesssim \varepsilon^{-N_4},\quad \Vert u_{1,\varepsilon}\Vert_{H^{s}} \lesssim \varepsilon^{-N_5}.
    \end{equation*}
    It follows from the energy estimate \eqref{Energy1 aux1+1 non 1}, that
    \begin{equation*}
        \Vert u_{\varepsilon}(t,\cdot)\Vert_2 \lesssim \varepsilon^{-2N_1-2N_2-\max \{N_3,N_4,N_5\}},
    \end{equation*}
    uniformly in $t\in [0,T]$, which means that the net $(u_{\varepsilon})_{\varepsilon}$ is $C_2$-moderate.\\
    This concludes the proof.
\end{proof}
Here and in all what follows, we consider only the case when $g\equiv 1$.

\subsection{Uniqueness}
In this section, we want to prove the uniqueness of the very weak solution to the Cauchy problem \eqref{Equation sing g1} in both cases, either in the case when very weak solutions exist with the assumptions of Theorem \ref{Thm1 existence1} or in the case of Theorem \ref{Thm2 existence1}.
Roughly speaking, we understand the uniqueness of the very weak solution to the Cauchy problem \eqref{Equation sing g1}, in the sense that negligible changes in the regularisations of the equation coefficients and initial data, lead to negligible changes in the corresponding very weak solutions. More precisely,

\begin{defn}[\textbf{Uniqueness}]\label{Defn1 uniqueness}
    Let $g\equiv 1$. We say that the Cauchy problem \eqref{Equation sing g1} has a unique very weak solution, if for all families of regularisations 
     $(m_{\varepsilon})_{\varepsilon}$, $(\Tilde{m}_{\varepsilon})_{\varepsilon}$ and $(b_{\varepsilon})_{\varepsilon}$, $(\Tilde{b}_{\varepsilon})_{\varepsilon}$ for the equation coefficients $m$ and $b$, and for all families of regularisations $(f_{\varepsilon})_{\varepsilon}$, $(\Tilde{f}_{\varepsilon})_{\varepsilon}$ for the source term $f$, and families of regularisations $(u_{0,\varepsilon})_{\varepsilon}$, $(\Tilde{u}_{0,\varepsilon})_{\varepsilon}$ and  $(u_{1,\varepsilon})_{\varepsilon}$, $(\Tilde{u}_{1,\varepsilon})_{\varepsilon}$ for the Cauchy data $u_0$ and $u_1$, such that the nets
    $(m_{\varepsilon}-\Tilde{m}_{\varepsilon})_{\varepsilon}$, $(b_{\varepsilon}-\Tilde{b}_{\varepsilon})_{\varepsilon}$,  $(f_{\varepsilon}-\Tilde{f}_{\varepsilon})_{\varepsilon}$ and $(u_{0,\varepsilon}-\Tilde{u}_{0,\varepsilon})_{\varepsilon}$,$(u_{1,\varepsilon}-\Tilde{u}_{1,\varepsilon})_{\varepsilon}$ are $\big\{L^{\infty}(\mathbb{R}^d) ,L^{\infty}(\mathbb{R}^d),L^2(\mathbb{R}^d),H^{s}(\mathbb{R}^d),L^2(\mathbb{R}^d)\big\}$-negligible, it follows that the net
    \begin{equation*}
        \big(u_{\varepsilon}(t,\cdot)-\Tilde{u}_{\varepsilon}(t,\cdot)\big)_{\varepsilon}
    \end{equation*}
    is $L^2(\mathbb{R}^d)$-negligible, for all $t\in [0,T]$, where $(u_{\varepsilon})_{\varepsilon}$ and $(\Tilde{u}_{\varepsilon})_{\varepsilon}$ are the families of solutions to the regularised Cauchy problems
    \begin{equation}\label{Regularised equation for u}
        \left\lbrace
        \begin{array}{l}
        \partial_{t}^2u_{\varepsilon}(t,x) + (-\Delta)^{s}u_{\varepsilon}(t,x) + m_{\varepsilon}(x)u_{\varepsilon}(t,x) + b_{\varepsilon}(x)\partial_{t}u_{\varepsilon}(t,x)=f_{\varepsilon}(t,x), \\ \hfill (t,x)\in [0,T]\times\mathbb{R}^d,\\
        u_{\varepsilon}(0,x)=u_{0,\varepsilon}(x),\quad \partial_{t}u_{\varepsilon}(0,x)=u_{1,\varepsilon}(x), \quad x\in\mathbb{R}^d,
        \end{array}
        \right.
    \end{equation}
    and
    \begin{equation}\label{Regularised equation for u tilde}
        \left\lbrace
        \begin{array}{l}
        \partial_{t}^2\Tilde{u}_{\varepsilon}(t,x) + (-\Delta)^{s}\Tilde{u}_{\varepsilon}(t,x) + \Tilde{m}_{\varepsilon}(x)\Tilde{u}_{\varepsilon}(t,x) + \Tilde{b}_{\varepsilon}(x)\partial_{t}\Tilde{u}_{\varepsilon}(t,x)=\Tilde{f}_{\varepsilon}(t,x), \\ \hfill (t,x)\in [0,T]\times\mathbb{R}^d,\\
        \Tilde{u}_{\varepsilon}(0,x)=\Tilde{u}_{0,\varepsilon}(x),\quad \partial_{t}\Tilde{u}_{\varepsilon}(0,x)=\Tilde{u}_{1,\varepsilon}(x), \quad x\in\mathbb{R}^d,
        \end{array}
        \right.
    \end{equation}
    respectively.
\end{defn}

\begin{thm}\label{Thm1 uniqueness}
    Let $g\equiv 1$. Assume that $m,b \geq 0$, in the sense that there regularisations are non-negative. Under the conditions of Theorem \ref{Thm1 existence1}, the very weak solution to the Cauchy problem \eqref{Equation sing g1} is unique.
\end{thm}

\begin{proof}
    Let $(u_{\varepsilon})_{\varepsilon}$ and $(\Tilde{u}_{\varepsilon})_{\varepsilon}$ be very weak solutions to \eqref{Equation sing g1}, and assume that the nets $(m_{\varepsilon}-\Tilde{m}_{\varepsilon})_{\varepsilon}$, $(b_{\varepsilon}-\Tilde{b}_{\varepsilon})_{\varepsilon}$, $(f_{\varepsilon}-\Tilde{f}_{\varepsilon})_{\varepsilon}$ and $(u_{0,\varepsilon}-\Tilde{u}_{0,\varepsilon})_{\varepsilon}$, $(u_{1,\varepsilon}-\Tilde{u}_{1,\varepsilon})_{\varepsilon}$  are $L^{\infty}(\mathbb{R}^d)$, $L^{\infty}(\mathbb{R}^d)$, $L^2(\mathbb{R}^d)$, $H^{s}(\mathbb{R}^d)$, $L^2(\mathbb{R}^d)$-negligible, respectively. The function $U_{\varepsilon}(t,x)$ defined by
    \begin{equation*}
        U_{\varepsilon}(t,x):=u_{\varepsilon}(t,x)-\Tilde{u}_{\varepsilon}(t,x),
    \end{equation*}
    satisfies
    \begin{equation}\label{Equation1 U}
        \left\lbrace
        \begin{array}{l}
        \partial_{t}^2U_{\varepsilon}(t,x) + (-\Delta)^{s}U_{\varepsilon}(t,x) + m_{\varepsilon}(x)U_{\varepsilon}(t,x) + b_{\varepsilon}(x)\partial_{t}U_{\varepsilon}(t,x)=h_{\varepsilon}(t,x),\\
        U_{\varepsilon}(0,x)=(u_{0,\varepsilon}-\Tilde{u}_{0,\varepsilon})(x),\quad \partial_{t}U_{\varepsilon}(0,x)=(u_{1,\varepsilon}-\Tilde{u}_{1,\varepsilon})(x),
        \end{array}
        \right.
    \end{equation}
    for $(t,x)\in [0,T]\times\mathbb{R}^d$, where
    \begin{align}
        h_{\varepsilon}(t,x)&:=\big(f_{\varepsilon}(t,x)-\Tilde{f}_{\varepsilon}(t,x)\big) +\big(\Tilde{m}_{\varepsilon}(x)-m_{\varepsilon}(x)\big)\Tilde{u}_{\varepsilon}(t,x)   \nonumber \\ &\quad+ \big(\Tilde{b}_{\varepsilon}(x)-b_{\varepsilon}(x)\big)\partial_{t}\Tilde{u}_{\varepsilon}(t,x).\nonumber
    \end{align}
    According to Duhamel's principle (Theorem \ref{Thm Duhamel}), the solution to \eqref{Equation1 U} has the following representation
    \begin{equation}\label{Duhamel1 solution}
        U_{\varepsilon}(t,x) = W_{\varepsilon}(t,x) + \int_{0}^{t}V_{\varepsilon}(t,x;\tau)\d \tau,
    \end{equation}
    where $W_{\varepsilon}(t,x)$ is the solution to the homogeneous problem
    \begin{equation}\label{Equation1 W 1}
        \bigg\{
        \begin{array}{l}
        \partial_{t}^2W_{\varepsilon}(t,x) + (-\Delta)^{s}W_{\varepsilon}(t,x) + m_{\varepsilon}(x)W_{\varepsilon}(t,x) + b_{\varepsilon}(x)\partial_{t}W_{\varepsilon}(t,x)=0,\\
        W_{\varepsilon}(0,x)=(u_{0,\varepsilon}-\Tilde{u}_{0,\varepsilon})(x),\quad \partial_{t}W_{\varepsilon}(0,x)=(u_{1,\varepsilon}-\Tilde{u}_{1,\varepsilon})(x),
        \end{array}
    \end{equation}
    for $(t,x)\in [0,T]\times\mathbb{R}^d$, and $V_{\varepsilon}(t,x;\tau)$ solves
    \begin{equation}\label{Equation1 V}
        \bigg\{
        \begin{array}{l}
        \partial_{t}^2V_{\varepsilon}(t,x;\tau) + (-\Delta)^{s}V_{\varepsilon}(t,x;\tau) + m_{\varepsilon}(x)V_{\varepsilon}(t,x;\tau) + b_{\varepsilon}(x)\partial_{t}V_{\varepsilon}(t,x;\tau)=0,\\
        V_{\varepsilon}(\tau,x;\tau)=0,\quad \partial_{t}V_{\varepsilon}(\tau,x;\tau)=h_{\varepsilon}(\tau,x),
        \end{array}
    \end{equation}
    for $(t,x)\in [\tau,T]\times\mathbb{R}^d$ and $\tau \in [0,T]$.
    By taking the $L^2$-norm in both sides of \eqref{Duhamel1 solution} and using Minkowski's integral inequality, we get
    \begin{equation}\label{Duhamel1 solution estimate}
        \Vert U_{\varepsilon}(t,\cdot)\Vert_{L^2} \leq \Vert W_{\varepsilon}(t,\cdot)\Vert_{L^2} + \int_{0}^{t}\Vert V_{\varepsilon}(t,\cdot;\tau)\Vert_{L^2}\d\tau.
    \end{equation}
    The energy estimate \eqref{Energy estimate g1} allows us to control $\Vert W_{\varepsilon}(t,\cdot)\Vert_{L^2}$ and $\Vert V_{\varepsilon}(t,\cdot;\tau)\Vert_{L^2}$. We have
    \begin{equation*}
        \Vert W_{\varepsilon}(t,\cdot)\Vert_{L^2} \lesssim \Big(2+ \Vert m_{\varepsilon}\Vert_{L^{\infty}}\Big) \Big(1+ \Vert b_{\varepsilon} \Vert_{L^{\infty}} +  \Vert m_{\varepsilon} \Vert_{L^{\infty}}^{\frac{1}{2}}\Big) \bigg[\Vert u_{0,\varepsilon}-\Tilde{u}_{0,\varepsilon}\Vert_{H^s} + \Vert u_{1,\varepsilon}-\Tilde{u}_{1,\varepsilon}\Vert_{L^2}\bigg],
    \end{equation*}
    and
    \begin{equation*}
        \Vert V_{\varepsilon}(t,\cdot;\tau)\Vert_{L^2} \lesssim\Big(2 + \Vert m_{\varepsilon}\Vert_{L^{\infty}}\Big) \Big(1+ \Vert b_{\varepsilon} \Vert_{L^{\infty}} +  \Vert m_{\varepsilon} \Vert_{L^{\infty}}^{\frac{1}{2}}\Big)\bigg[\Vert h_{\varepsilon}(\tau,\cdot)\Vert_{L^2}\bigg].
    \end{equation*}
    By taking into consideration that $t\in[0,T]$, it follows from \eqref{Duhamel1 solution estimate} that
    \begin{align}\label{Estimate U 1}
        \Vert U_{\varepsilon}(t,\cdot)\Vert_{L^2} &\lesssim \Big(2+ \Vert m_{\varepsilon}\Vert_{L^{\infty}}\Big) \Big(1+ \Vert b_{\varepsilon} \Vert_{L^{\infty}} +  \Vert m_{\varepsilon} \Vert_{L^{\infty}}^{\frac{1}{2}}\Big) \\ &\quad \times \bigg[\Vert u_{0,\varepsilon}-\Tilde{u}_{0,\varepsilon}\Vert_{H^s} +
        \Vert u_{1,\varepsilon}-\Tilde{u}_{1,\varepsilon}\Vert_{L^2} + \int_{0}^{T}\Vert h_{\varepsilon}(\tau,\cdot)\Vert_{L^2}\bigg] ,\nonumber
    \end{align}
    where $\Vert h_{\varepsilon}(\tau,\cdot)\Vert_{L^2}$ is estimated as follows,
    \begin{align}\label{Estimate f_epsilon}
        \Vert h_{\varepsilon}(\tau,\cdot)\Vert_{L^2} & \leq \Vert\big(f_{\varepsilon}(\tau,\cdot)-\Tilde{f}_{\varepsilon}(\tau,\cdot)\big) \Vert_{L^2} +\Vert \big(\Tilde{m}_{\varepsilon}(\cdot)-m_{\varepsilon}(\cdot)\big)\Tilde{u}_{\varepsilon}(\tau,\cdot) \Vert_{L^2}    \\ &\quad+\Vert \big(\Tilde{b}_{\varepsilon}(\cdot)-b_{\varepsilon}(\cdot)\big)\partial_{t}\Tilde{u}_{\varepsilon}(\tau,\cdot)\Vert_{L^2} \nonumber\\
        & \leq \Vert f_{\varepsilon}(\tau,\cdot) - \Tilde{f}_{\varepsilon}(\tau,\cdot)\Vert_{L^2} +\Vert \Tilde{m}_{\varepsilon} - m_{\varepsilon}\Vert_{L^{\infty}}\Vert \Tilde{u}_{\varepsilon}(\tau,\cdot)\Vert_{L^2} \nonumber \\ &\quad + \Vert \Tilde{b}_{\varepsilon} - b_{\varepsilon}\Vert_{L^{\infty}}\Vert \partial_{t}\Tilde{u}_{\varepsilon}(\tau,\cdot)\Vert_{L^2} \nonumber.
    \end{align}
    On the one hand, the nets 
    $(m_{\varepsilon})_{\varepsilon}$, $(b_{\varepsilon})_{\varepsilon}$ are $L^{\infty}$-moderate by assumption, and the net $(\Tilde{u}_{\varepsilon})_{\varepsilon}$ is $C_1$-moderate being a very weak solution to \eqref{Regularised equation for u tilde}. From the other hand, the nets $(m_{\varepsilon}-\Tilde{m}_{\varepsilon})_{\varepsilon}$, $(b_{\varepsilon}-\Tilde{b}_{\varepsilon})_{\varepsilon}$, $(f_{\varepsilon}-\Tilde{f}_{\varepsilon})_{\varepsilon}$ and $(u_{0,\varepsilon}-\Tilde{u}_{0,\varepsilon})_{\varepsilon}$, $(u_{1,\varepsilon}-\Tilde{u}_{1,\varepsilon})_{\varepsilon}$ are $L^{\infty}(\mathbb{R}^d)$,$L^{\infty}(\mathbb{R}^d)$, $L^2(\mathbb{R}^d)$, $H^{s}(\mathbb{R}^d)$, $L^{2}(\mathbb{R}^d)$-negligible. It follows from \eqref{Estimate U 1} combined with \eqref{Estimate f_epsilon} that
    \begin{equation*}
        \Vert U_{\varepsilon}(t,\cdot)\Vert_{L^2} \lesssim \varepsilon^{k},
    \end{equation*}
    for all $k>0$, showing the uniqueness of the very weak solution.    
\end{proof}
The analogue to Definition \ref{Defn1 uniqueness} and Theorem \ref{Thm1 uniqueness} in the case when $d>2s$ with Theorem \ref{Thm2 existence1}'s background, read:

\begin{defn}\label{Defn2 uniqueness}
    Let $g\equiv 1$. We say that the Cauchy problem \eqref{Equation sing g1} has a unique very weak solution, if for all families of regularisations  $(m_{\varepsilon})_{\varepsilon}$, $(\Tilde{m}_{\varepsilon})_{\varepsilon}$ and $(b_{\varepsilon})_{\varepsilon}$, $(\Tilde{b}_{\varepsilon})_{\varepsilon}$ for the equation coefficients  $m$ and $b$, and for families of regularisation $(f_{\varepsilon})_{\varepsilon}$, $(\Tilde{f}_{\varepsilon})_{\varepsilon}$ for the source term, and families of regularisations $(u_{0,\varepsilon})_{\varepsilon}$, $(\Tilde{u}_{0,\varepsilon})_{\varepsilon}$ and  $(u_{1,\varepsilon})_{\varepsilon}$, $(\Tilde{u}_{1,\varepsilon})_{\varepsilon}$ for the Cauchy data $u_0$ and $u_1$, such that the nets $(m_{\varepsilon}-\Tilde{m}_{\varepsilon})_{\varepsilon}$, $(b_{\varepsilon}-\Tilde{b}_{\varepsilon})_{\varepsilon}$,  $(f_{\varepsilon}-\Tilde{f}_{\varepsilon})_{\varepsilon}$ and $(u_{0,\varepsilon}-\Tilde{u}_{0,\varepsilon})_{\varepsilon}$, $(u_{1,\varepsilon}-\Tilde{u}_{1,\varepsilon})_{\varepsilon}$ are $\big\{(L^{\frac{d}{s}}(\mathbb{R}^d)\cap L^{\frac{d}{2s}}(\mathbb{R}^d)),L^{\frac{d}{s}}(\mathbb{R}^d),L^2 (\mathbb{R}^d), H^{2s}(\mathbb{R}^d), H^{s}(\mathbb{R}^d)\big\}$-negligible, it follows that the net \\ $\big(u_{\varepsilon}(t,\cdot)-\Tilde{u}_{\varepsilon}(t,\cdot)\big)_{\varepsilon\in (0,1]}$, is $L^2(\mathbb{R}^d)$-negligible for all $t\in [0,T]$, where $(u_{\varepsilon})_{\varepsilon}$ and $(\Tilde{u}_{\varepsilon})_{\varepsilon}$ are the families of solutions to the corresponding regularised Cauchy problems.
\end{defn}

\begin{thm}\label{Thm2 uniqueness}
    Let $d>2s$. Assume that $g\equiv 1$ and $m,b \geq 0$. With the assumptions of Theorem \ref{Thm2 existence1}, the very weak solution to the Cauchy problem \eqref{Equation sing g1} is unique.
\end{thm}

\begin{proof}
 Let $(u_{\varepsilon})_{\varepsilon}$ and $(\Tilde{u}_{\varepsilon})_{\varepsilon}$ be very weak solutions to \eqref{Equation sing g1}, and assume that the nets $(m_{\varepsilon}-\Tilde{m}_{\varepsilon})_{\varepsilon}$, $(b_{\varepsilon}-\Tilde{b}_{\varepsilon})_{\varepsilon}$, $(f_{\varepsilon}-\Tilde{f}_{\varepsilon})_{\varepsilon}$ and $(u_{0,\varepsilon}-\Tilde{u}_{0,\varepsilon})_{\varepsilon}$,$(u_{1,\varepsilon}-\Tilde{u}_{1,\varepsilon})_{\varepsilon}$  are $L^{\frac{d}{s}}(\mathbb{R}^d)\cap L^{\frac{d}{2s}}(\mathbb{R}^d)$, $L^{\frac{d}{s}}(\mathbb{R}^d)$,$L^2(\mathbb{R}^d)$, $H^{2s}(\mathbb{R}^d)$, $H^{s}(\mathbb{R}^d)$-negligible, respectively. Let us denote by $U_{\varepsilon}(t,x)$ the function defined by
    \begin{equation*}
        U_{\varepsilon}(t,x):=u_{\varepsilon}(t,x)-\Tilde{u}_{\varepsilon}(t,x).
    \end{equation*}
    It satisfies
    \begin{equation}\label{Equation1 U2}
        \left\lbrace
        \begin{array}{l}
        \partial_{t}^2U_{\varepsilon}(t,x) + (-\Delta)^{s}U_{\varepsilon}(t,x) + m_{\varepsilon}(x)U_{\varepsilon}(t,x) + b_{\varepsilon}(x)\partial_{t}U_{\varepsilon}(t,x)=h_{\varepsilon}(t,x),\\
        U_{\varepsilon}(0,x)=(u_{0,\varepsilon}-\Tilde{u}_{0,\varepsilon})(x),\quad \partial_{t}U_{\varepsilon}(0,x)=(u_{1,\varepsilon}-\Tilde{u}_{1,\varepsilon})(x),
        \end{array}
        \right.
    \end{equation}
    for $(t,x)\in [0,T]\times\mathbb{R}^d$, where
    \begin{align}
        h_{\varepsilon}(t,x)&:=\big(f_{\varepsilon}(t,x)-\Tilde{f}_{\varepsilon}(t,x)\big) +\big(\Tilde{m}_{\varepsilon}(x)-m_{\varepsilon}(x)\big)\Tilde{u}_{\varepsilon}(t,x)   \nonumber \\ &\quad+ \big(\Tilde{b}_{\varepsilon}(x)-b_{\varepsilon}(x)\big)\partial_{t}\Tilde{u}_{\varepsilon}(t,x).\nonumber
    \end{align} By repeating the same reasoning as in Theorem \ref{Thm1 uniqueness}, and using Duhamel's principle, the solution to \eqref{Equation1 U2} has the following representation
    \begin{equation}\label{Duhamel1 solution 1}
        U_{\varepsilon}(t,x) = W_{\varepsilon}(t,x) + \int_{0}^{t}V_{\varepsilon}(t,x;\tau)\d \tau.
    \end{equation} 
    By taking the $L^2$-norm in both sides of \eqref{Duhamel1 solution 1} and using Minkowski's integral inequality, we get
    \begin{equation}\label{Duhamel1 solution estimate 1}
        \Vert U_{\varepsilon}(t,\cdot)\Vert_{L^2} \leq \Vert W_{\varepsilon}(t,\cdot)\Vert_{L^2} + \int_{0}^{t}\Vert V_{\varepsilon}(t,\cdot;\tau)\Vert_{L^2}\d\tau,
    \end{equation}
 where $\Vert W_{\varepsilon}(t,\cdot)\Vert_{L^2}$ and $\Vert V_{\varepsilon}(t,\cdot;\tau)\Vert_{L^2}$ are estimated by using \eqref{Energy estimate1} as follows:
    \begin{align*}
        \Vert W_{\varepsilon}(t,\cdot)\Vert_{L^2} \lesssim 
        &\Big(1+\Vert m_{\varepsilon}\Vert_{L^{\frac{d}{s}}}\Big)\Big(1+\Vert m_{\varepsilon}\Vert_{L^{\frac{d}{2s}}}\Big)\\
        & \times \Big(1+\Vert b_{\varepsilon}\Vert_{L^{\frac{d}{s}}}\Big)^2 \bigg[\Vert u_{0,\varepsilon}-\Tilde{u}_{0,\varepsilon}\Vert_{H^{2s}} + \Vert u_{1,\varepsilon}-\Tilde{u}_{1,\varepsilon}\Vert_{H^s}\bigg],
    \end{align*}
    and
    \begin{equation*}
        \Vert V_{\varepsilon}(t,\cdot;\tau)\Vert_{L^2} \lesssim \Big(1+\Vert m_{\varepsilon}\Vert_{L^{\frac{d}{s}}}\Big)\Big(1+\Vert m_{\varepsilon}\Vert_{L^{\frac{d}{2s}}}\Big)\Big(1+\Vert b_{\varepsilon}\Vert_{L^{\frac{d}{s}}}\Big)^2\bigg[\Vert h_{\varepsilon}(\tau,\cdot)\Vert_{L^2}\bigg].
    \end{equation*}
    By taking into consideration that $t\in[0,T]$, it follows from \eqref{Duhamel1 solution estimate 1} that
    \begin{align}\label{Estimate U 1+1}
        \Vert U_{\varepsilon}(t,\cdot)\Vert_{L^2} &\lesssim \Big(1+\Vert m_{\varepsilon}\Vert_{L^{\frac{d}{s}}}\Big)\Big(1+\Vert m_{\varepsilon}\Vert_{L^{\frac{d}{2s}}}\Big)\Big(1+\Vert b_{\varepsilon}\Vert_{L^{\frac{d}{s}}}\Big)^2 \\ &\quad \times \bigg[\Vert u_{0,\varepsilon}-\Tilde{u}_{0,\varepsilon}\Vert_{H^{2s}} +
        \Vert u_{1,\varepsilon}-\Tilde{u}_{1,\varepsilon}\Vert_{H^s} + \int_{0}^{T}\Vert h_{\varepsilon}(\tau,\cdot)\Vert_{L^2}\bigg] ,\nonumber
    \end{align}
    where $\Vert h_{\varepsilon}(\tau,\cdot)\Vert_{L^2}$ is estimated as follows
    \begin{align}
        \Vert h_{\varepsilon}(\tau,\cdot)\Vert_{L^2} & \leq \Vert\big(f_{\varepsilon}(\tau,\cdot)-\Tilde{f}_{\varepsilon}(\tau,\cdot)\big) \Vert_{L^2} +\Vert \big(\Tilde{m}_{\varepsilon}(\cdot)-m_{\varepsilon}(\cdot)\big)\Tilde{u}_{\varepsilon}(\tau,\cdot) \Vert_{L^2} \nonumber   \\ &\quad +\Vert \big(\Tilde{b}_{\varepsilon}(\cdot)-b_{\varepsilon}(\cdot)\big)\partial_{t}\Tilde{u}_{\varepsilon}(\tau,\cdot)\Vert_{L^2} \nonumber\\
        & \leq \Vert f_{\varepsilon}(\tau,\cdot) - \Tilde{f}_{\varepsilon}(\tau,\cdot)\Vert_{L^2} +\Vert \Tilde{m}_{\varepsilon} - m_{\varepsilon}\Vert_{L^p}\Vert \Tilde{u}_{\varepsilon}(\tau,\cdot)\Vert_{L^q} \nonumber \\ &\quad + \Vert \Tilde{b}_{\varepsilon} - b_{\varepsilon}\Vert_{L^p}\Vert \partial_{t}\Tilde{u}_{\varepsilon}(\tau,\cdot)\Vert_{L^q} \nonumber.
    \end{align}
      The second and third terms on the right hand side are estimated by using H\"older's inequality for $1<p<\infty$, such that $\frac{1}{p} + \frac{1}{q} = \frac{1}{2}$. Then, if we choose $q=\frac{2d}{d-2s}$ and consequently $p=\frac{d}{s}$, it follows from Proposition \ref{Prop. Sobolev estimate} that 
    \begin{equation*}
      \Vert \Tilde{u}_{\varepsilon}\Vert_{L^q} \lesssim \Vert (-\Delta)^{\frac{s}{2}} \Tilde{u}_{\varepsilon}(\cdot)\Vert_{L^2}\leq \Vert \Tilde{u}_{\varepsilon}\Vert_{H^s},
    \end{equation*}
    and 
    \begin{equation*}
        \Vert \partial_t \Tilde{u}_{\varepsilon}\Vert_{L^q} \lesssim \Vert (-\Delta)^{\frac{s}{2}} \partial_t \Tilde{u}_{\varepsilon}(\cdot)\Vert_{L^2}\leq \Vert \partial_t \Tilde{u}_{\varepsilon}\Vert_{H^s}.
    \end{equation*}
    We get
    \begin{align}\label{Estimate f_epsilon 1}
        \Vert h_{\varepsilon}(\tau,\cdot)\Vert_{L^2}
        & \leq \Vert f_{\varepsilon}(\tau,\cdot) - \Tilde{f}_{\varepsilon}(\tau,\cdot)\Vert_{L^2} +\Vert \Tilde{m}_{\varepsilon} - m_{\varepsilon}\Vert_{L^{\frac{d}{s}}}\Vert \Tilde{u}_{\varepsilon}(\tau,\cdot)\Vert_{H^s}  \\ &\quad + \Vert \Tilde{b}_{\varepsilon} - b_{\varepsilon}\Vert_{L^{\frac{d}{s}}}\Vert \partial_{t}\Tilde{u}_{\varepsilon}(\tau,\cdot)\Vert_{H^s} \nonumber.
    \end{align}
    On the one hand, the nets 
    $(m_{\varepsilon})_{\varepsilon}$, $(b_{\varepsilon})_{\varepsilon}$ are $L^{\frac{d}{s}}\cap L^{\frac{d}{2s}}$, $L^{\frac{d}{s}}$-moderate by assumption, and the net $(\Tilde{u}_{\varepsilon})_{\varepsilon}$ is $C_2$-moderate being a very weak solution to \eqref{Regularised equation for u tilde}. From the other hand, the nets $(m_{\varepsilon}-\Tilde{m}_{\varepsilon})_{\varepsilon}$, $(b_{\varepsilon}-\Tilde{b}_{\varepsilon})_{\varepsilon}$, $(f_{\varepsilon}-\Tilde{f}_{\varepsilon})_{\varepsilon}$ and $(u_{0,\varepsilon}-\Tilde{u}_{0,\varepsilon})_{\varepsilon}$, $(u_{1,\varepsilon}-\Tilde{u}_{1,\varepsilon})_{\varepsilon}$ are $L^{\frac{d}{s}}(\mathbb{R}^d)\cap L^{\frac{d}{2s}}(\mathbb{R}^d)$, $L^{\frac{d}{s}}(\mathbb{R}^d)$,$L^{2}(\mathbb{R}^d)$, $H^{2s}(\mathbb{R}^d)$, $H^{s}(\mathbb{R}^d)$-negligible. It follows from \eqref{Estimate U 1+1} combined with \eqref{Estimate f_epsilon 1} that
    \begin{equation*}
        \Vert U_{\varepsilon}(t,\cdot)\Vert_{L^2} \lesssim \varepsilon^{k},
    \end{equation*}
    for all $k>0$. The uniqueness of the very weak solution follows, ending the proof. 
\end{proof}

\section{Coherence with classical theory}
 Now, let $g\equiv 1$ and the coefficients $m,b,f$, and the Cauchy data $u_0$ and $u_1$ be smooth enough in such a way that a classical solution to 
 \begin{equation}\label{Equation coherence}
        \left\lbrace
        \begin{array}{l}
        u_{tt}(t,x) + (-\Delta)^{s}u(t,x) + m(x)u(t,x) + b(x)u_{t}(t,x)=f(t,x), \\ 
        \hfill (t,x)\in [0,T]\times\mathbb{R}^d,\\
        u(0,x)=u_{0}(x),\quad u_{t}(0,x)=u_{1}(x), \quad x\in\mathbb{R}^d,
        \end{array}
        \right.
\end{equation}
 exists (in the sense of Theorem \ref{thm non-homogeneous 1}). For these coefficients and initial data, the concept of very weak solutions applies (according to Proposition \ref{coherence 1}) and a very weak solution exists. The question to be answered here is: does the very weak solution recapture the classical one?

\begin{thm}\label{Thm1 coherence}
    Let $\psi$ be a Friedrichs-mollifier. Assume $m,b\in L^{\infty}(\mathbb{R}^d)$ be non-negative, $f\in L^2(\mathbb R^d)$ and suppose that $u_0 \in H^{s}(\mathbb{R}^d)$ and $u_1 \in L^{2}(\mathbb{R}^d)$, in such a way that a classical solution to \eqref{Equation coherence} exists. Then, for any regularising families $(m_{\varepsilon})_{\varepsilon}=(m\ast\psi_{\varepsilon})_{\varepsilon}$ and $(b_{\varepsilon})_{\varepsilon}=(b\ast\psi_{\varepsilon})_{\varepsilon}$ for the equation coefficients, satisfying
    \begin{equation}\label{approx.condition1}
        \lVert m_{\varepsilon} - m\rVert_{L^{\infty}} \rightarrow 0,\quad\text{and}\quad \lVert b_{\varepsilon} - b\rVert_{L^{\infty}} \rightarrow 0,
    \end{equation} 
    and any regularising families $(f_{\varepsilon})_{\varepsilon}=(f\ast\psi_{\varepsilon})_{\varepsilon}$ for the source term, and any regularising families $(u_{0,\varepsilon})_{\varepsilon}=(u_{0}\ast\psi_{\varepsilon})_{\varepsilon}$ and $(u_{1,\varepsilon})_{\varepsilon}=(u_{1}\ast\psi_{\varepsilon})_{\varepsilon}$ for the initial data, the net $(u_{\varepsilon})_{\varepsilon}$ converges to the classical solution of the Cauchy problem (\ref{Equation coherence}) in $L^{2}$ as $\varepsilon \rightarrow 0$.
\end{thm}

\begin{proof}
    Let $(u_{\varepsilon})_{\varepsilon}$ be the very weak solution given by Theorem \ref{Thm1 existence1} and $u$ the classical one, as in Theorem \ref{thm non-homogeneous 1}. The classical solution satisfies
    \begin{equation}\label{Equation coherence classic}
        \left\lbrace
        \begin{array}{l}
        u_{tt}(t,x) + (-\Delta)^{s}u(t,x) + m(x)u(t,x) + b(x)u_{t}(t,x)=f(t,x), \\
        \hfill (t,x)\in [0,T]\times\mathbb{R}^d,\\
        u(0,x)=u_{0}(x),\quad u_{t}(0,x)=u_{1}(x), \quad x\in\mathbb{R}^d,
        \end{array}
        \right.
    \end{equation}
    and $(u_{\varepsilon})_{\varepsilon}$ solves
    \begin{equation}\label{Equation coherence VWS}
        \left\lbrace
        \begin{array}{l}
        \partial_{t}^2u_{\varepsilon}(t,x) + (-\Delta)^{s}u_{\varepsilon}(t,x) + m_{\varepsilon}(x)u_{\varepsilon}(t,x) + b_{\varepsilon}(x)\partial_{t}u_{\varepsilon}(t,x)=f_{\varepsilon}(t,x), \\ 
        \hfill (t,x)\in [0,T]\times\mathbb{R}^d,\\
        u_{\varepsilon}(0,x)=u_{0,\varepsilon}(x),\quad \partial_{t}u_{\varepsilon}(0,x)=u_{1,\varepsilon}(x), \quad x\in\mathbb{R}^d.
        \end{array}
        \right.
    \end{equation}
    By denoting $U_{\varepsilon}(t,x):=u_{\varepsilon}(t,x)-u(t,x)$, then $U_{\varepsilon}$ solves the Cauchy problem
    \begin{equation}\label{Equation coherence V}
        \left\lbrace
        \begin{array}{l}
        \partial_{t}^2U_{\varepsilon}(t,x) + (-\Delta)^{s}U_{\varepsilon}(t,x) + m_{\varepsilon}(x)U_{\varepsilon}(t,x) + b_{\varepsilon}(x)\partial_{t}U_{\varepsilon}(t,x)=\rho_{\varepsilon}(t,x),\\
        \hfill  (t,x)\in [0,T]\times\mathbb{R}^d,\\
        U_{\varepsilon}(0,x)=(u_{0,\varepsilon}-u_0)(x),\quad \partial_{t}U_{\varepsilon}(0,x)=(u_{1,\varepsilon}-u_1)(x),\, \, x\in \mathbb R^d,
        \end{array}
        \right.
    \end{equation}
    where
    \begin{equation}
        \rho_{\varepsilon}(t,x):=(f_{\varepsilon}(t,x)-f(t,x)) -\big(m_{\varepsilon}(x)-m(x)\big)u(t,x) - \big(b_{\varepsilon}(x)-b(x)\big)\partial_{t}u(t,x).
    \end{equation}
    Thanks to Duhamel's principle, $U_{\varepsilon}$ can be represented by
    \begin{equation}\label{Duhamel2 solution}
        U_{\varepsilon}(t,x) = W_{\varepsilon}(t,x) + \int_{0}^{t}V_{\varepsilon}(t,x;\tau)d\tau,
    \end{equation}
    where $W_{\varepsilon}(t,x)$ is the solution to the homogeneous problem
    \begin{equation}\label{Equation2 W}
        \bigg\{
        \begin{array}{l}
        \partial_{t}^2W_{\varepsilon}(t,x) + (-\Delta)^{s}W_{\varepsilon}(t,x) + m_{\varepsilon}(x)W_{\varepsilon}(t,x) + b_{\varepsilon}(x)\partial_{t}W_{\varepsilon}(t,x)=0,\\
        W_{\varepsilon}(0,x)=(u_{0,\varepsilon}-u_0)(x),\quad \partial_{t}W_{\varepsilon}(0,x)=(u_{1,\varepsilon}-u_1)(x),
        \end{array}
    \end{equation}
    for $(t,x)\in [0,T]\times\mathbb{R}^d$, and $V_{\varepsilon}(t,x;\tau)$ solves
    \begin{equation}\label{Equation2 V}
        \bigg\{
        \begin{array}{l}
        \partial_{t}^2V_{\varepsilon}(t,x;\tau) + (-\Delta)^{s}V_{\varepsilon}(t,x;\tau) + m_{\varepsilon}(x)V_{\varepsilon}(t,x;\tau) + b_{\varepsilon}(x)\partial_{t}V_{\varepsilon}(t,x;\tau)=0,\\
        V_{\varepsilon}(\tau,x;\tau)=0,\quad \partial_{t}V_{\varepsilon}(\tau,x;\tau)=\rho_{\varepsilon}(\tau,x),
        \end{array}
    \end{equation}
    for $(t,x)\in [\tau,T]\times\mathbb{R}^d$ and $\tau \in [0,T]$. We take the $L^2$-norm in \eqref{Duhamel2 solution} and we argue as in the proof of Theorem \ref{Thm1 uniqueness}. We obtain
    \begin{equation}\label{Duhamel2 solution estimate}
        \Vert U_{\varepsilon}(t,\cdot)\Vert_{L^2} \leq \Vert W_{\varepsilon}(t,\cdot)\Vert_{L^2} + \int_{0}^{t}\Vert V_{\varepsilon}(t,\cdot;\tau)\Vert_{L^2}d\tau,
    \end{equation}
    where
    \begin{equation*}
        \Vert W_{\varepsilon}(t,\cdot)\Vert_{L^2} \lesssim \Big(2+ \Vert m_{\varepsilon}\Vert_{L^{\infty}}\Big) \Big(1+ \Vert b_{\varepsilon} \Vert_{L^{\infty}} +  \Vert m_{\varepsilon} \Vert_{L^{\infty}}^{\frac{1}{2}}\Big)\bigg[\Vert u_{0,\varepsilon}-u_{0}\Vert_{H^s} + \Vert u_{1,\varepsilon}-u_{1,}\Vert_{L^2}\bigg],
    \end{equation*}
    and
    \begin{equation*}
        \Vert V_{\varepsilon}(t,\cdot;\tau)\Vert_{L^2} \lesssim \Big(2+ \Vert m_{\varepsilon}\Vert_{L^{\infty}}\Big) \Big(1+ \Vert b_{\varepsilon} \Vert_{L^{\infty}} +  \Vert m_{\varepsilon} \Vert_{L^{\infty}}^{\frac{1}{2}}\Big)\bigg[\Vert \rho_{\varepsilon}(\tau,\cdot)\Vert_{L^2}\bigg],
    \end{equation*}
    by the energy estimate \eqref{Energy estimate g1} from Lemma \ref{Lemma1}, and $\rho_{\varepsilon}$ is estimated by
    \begin{equation}\label{Estimate Theta_epsilon}
        \Vert \rho_{\varepsilon}(\tau,\cdot)\Vert_{L^2} \leq \Vert f_{\varepsilon}(\tau ,\cdot )- f(\tau ,\cdot )\Vert_{L^2} +\Vert m_{\varepsilon} - m\Vert_{L^{\infty}}\Vert u(t,\cdot)\Vert_{L^2} + \Vert b_{\varepsilon} - b\Vert_{L^{\infty}}\Vert \partial_{t}u(t,\cdot)\Vert_{L^2}.
    \end{equation}
    First, one observes that $\Vert m_{\varepsilon}\Vert_{L^{\infty}}<\infty$ and $\Vert b_{\varepsilon}\Vert_{L^{\infty}}<\infty$ by the fact that $m,b\in L^{\infty}(\mathbb{R}^d)$ and  $\Vert u(t,\cdot)\Vert_{L^2}$ and $\Vert \partial_{t}u(t,\cdot)\Vert_{L^2}$ are bounded as well, since $u$ is a classical solution to \eqref{Equation coherence}. This, together with
    \begin{equation*}
        \lVert m_{\varepsilon} - m\rVert_{L^{\infty}} \rightarrow 0,\quad \text{and}\quad \lVert b_{\varepsilon} - b\rVert_{L^{\infty}} \rightarrow 0,\quad \text{as } \varepsilon\rightarrow 0,
    \end{equation*}
    from the assumptions, and
    \begin{equation*}
        \lVert f_{\varepsilon}(t,\cdot) - f(t,\cdot)\rVert_{L^{2}} \rightarrow 0,\quad\text{and}\quad \Vert u_{0,\varepsilon}-u_{0}\Vert_{H^s} \rightarrow 0,\quad  \Vert u_{1,\varepsilon}-u_{1,}\Vert_{L^2} \rightarrow 0,\quad \text{as } \varepsilon\rightarrow 0,
    \end{equation*}
    shows that
    \begin{equation*}
        \Vert U_{\varepsilon}(t,\cdot)\Vert_{L^2} \rightarrow 0,\quad \text{as } \varepsilon\rightarrow 0,
    \end{equation*}
    uniformly in $t\in [0,T]$. This finishes the proof.    
\end{proof}
In the case when a classical solution exists in the sense of Theorem \ref{thm non-homogeneous 1} with $d>2s$, the coherence theorem reads as follows.

\begin{thm}\label{Thm2 coherence}
    Let $\psi$ be a Friedrichs-mollifier. Assume $(m,b)\in (L^{\frac{d}{s}}(\mathbb{R}^d)\cap L^{\frac{d}{2s}}(\mathbb{R}^d))\times L^{\frac{d}{s}}(\mathbb{R}^d)$ be non-negative, $f\in L^2(\mathbb R^d)$, and suppose that $u_0 \in H^{2s}(\mathbb{R}^d)$ and $u_1 \in H^{s}(\mathbb{R}^d)$. Then, for any regularising families $(m_{\varepsilon})_{\varepsilon}=(m\ast\psi_{\varepsilon})_{\varepsilon}$, $(b_{\varepsilon})_{\varepsilon}=(b\ast\psi_{\varepsilon})_{\varepsilon}$ and $(f_{\varepsilon})_{\varepsilon}=(f\ast\psi_{\varepsilon})_{\varepsilon}$ for the equation coefficients, and any regularising families $(u_{0,\varepsilon})_{\varepsilon}=(u_{0}\ast\psi_{\varepsilon})_{\varepsilon}$ and $(u_{1,\varepsilon})_{\varepsilon}=(u_{1}\ast\psi_{\varepsilon})_{\varepsilon}$ for the initial data, the net $(u_{\varepsilon})_{\varepsilon}$ converges to the classical solution of the Cauchy problem \eqref{Equation coherence} in $L^{2}$ as $\varepsilon \rightarrow 0$.
\end{thm}

\begin{proof}
 Let $(u_{\varepsilon})_{\varepsilon}$ be the very weak solution given by Theorem \ref{Thm2 existence1} and $u$ the classical one, as in Theorem \ref{thm non-homogeneous 1} with $d>2s$. By using Duhamel's principle and arguing as in the proof of Theorem \ref{Thm2 uniqueness}, we arrive at
    \begin{equation}\label{Duhamel2 solution estimate 1}
        \Vert (u_{\varepsilon}-u)(t,\cdot)\Vert_{L^2} \leq \Vert W_{\varepsilon}(t,\cdot)\Vert_{L^2} + \int_{0}^{t}\Vert V_{\varepsilon}(t,\cdot;\tau)\Vert_{L^2}d\tau,
    \end{equation}
    where
    \begin{align*}
        \Vert W_{\varepsilon}(t,\cdot)\Vert_{L^2} \lesssim \Big(1+\Vert m_{\varepsilon}\Vert_{L^{\frac{d}{s}}}\Big)&\Big(1+\Vert m_{\varepsilon}\Vert_{L^{\frac{d}{2s}}}\Big)\Big(1+\Vert b_{\varepsilon}\Vert_{L^{\frac{d}{s}}}\Big)^2\\ &\quad \times \bigg[\Vert u_{0,\varepsilon}-u_{0}\Vert_{H^{2s}} + \Vert u_{1,\varepsilon}-u_{1,}\Vert_{H^s}\bigg],
    \end{align*}
    and
    \begin{equation*}
        \Vert V_{\varepsilon}(t,\cdot;\tau)\Vert_{L^2} \lesssim \Big(1+\Vert m_{\varepsilon}\Vert_{L^{\frac{d}{s}}}\Big)\Big(1+\Vert m_{\varepsilon}\Vert_{L^{\frac{d}{2s}}}\Big)\Big(1+\Vert b_{\varepsilon}\Vert_{L^{\frac{d}{s}}}\Big)^2\bigg[\Vert \rho_{\varepsilon}(\tau,\cdot)\Vert_{L^2}\bigg],
    \end{equation*}
    by the energy estimate \eqref{Energy estimate1}, from Lemma \ref{Lemma1}, and $\rho_{\varepsilon}$ is estimated by
    \begin{equation}\label{Estimate Theta_epsilon 1}
        \Vert \rho_{\varepsilon}(\tau,\cdot)\Vert_{L^2} \leq \Vert f_{\varepsilon}(\tau ,\cdot )- f(\tau ,\cdot )\Vert_{L^2} +\Vert m_{\varepsilon} - m\Vert_{L^{\frac{d}{s}}}\Vert u(t,\cdot)\Vert_{H^{s}} + \Vert b_{\varepsilon} - b\Vert_{L^{\frac{d}{s}}}\Vert \partial_{t}u(t,\cdot)\Vert_{H^s}.
    \end{equation}
    where the second and third terms on the right hand side are estimated by using H\"older's inequality (see Theorem \ref{Holder inequality}) for $1<p<\infty$, such that $\frac{1}{p} + \frac{1}{q} = \frac{1}{2}$ combined with the fractional Sobolev inequality (Proposition \ref{Prop. Sobolev estimate}).\\
    First, one observes that $\Vert m_{\varepsilon}\Vert_{L^{\frac{d}{s}}}<\infty$, $\Vert m_{\varepsilon}\Vert_{L^{\frac{d}{2s}}}<\infty $, and $\Vert b_{\varepsilon}\Vert_{L^{\frac{d}{s}}}<\infty$ by the fact that $(m,b)\in \big(L^{\frac{d}{s}}(\mathbb{R}^d)\cap L^{\frac{d}{2s}}(\mathbb{R}^d)\big)\times \big(L^{\frac{d}{s}}(\mathbb{R}^d)\big)$ and  $\Vert u(t,\cdot)\Vert_{H^{s}}$ and $\Vert \partial_{t}u(t,\cdot)\Vert_{H^{s}}$ are bounded as well, since $u$ is a classical solution to \eqref{Equation coherence}. This, together with
    \begin{equation*}
        \lVert m_{\varepsilon} - m\rVert_{L^{\frac{d}{s}}} \rightarrow 0,\quad\text{and}\quad \lVert b_{\varepsilon} - b\rVert_{L^{\frac{d}{s}}} \rightarrow 0,\quad\text{as } \varepsilon\rightarrow 0,
    \end{equation*}
    and
    \begin{equation*}
        \lVert f_{\varepsilon}(t,\cdot) - f(t,\cdot)\rVert_{L^{2}} \rightarrow 0,\, \Vert u_{0,\varepsilon}-u_{0}\Vert_{H^{2s}} \rightarrow 0,\,  \Vert u_{1,\varepsilon}-u_{1,}\Vert_{H^s} \rightarrow 0,\, \text{as } \varepsilon\rightarrow 0,
    \end{equation*}
    shows that
    \begin{equation*}
        \Vert (u_{\varepsilon}-u)(t,\cdot)\Vert_{L^2} \rightarrow 0,\quad \text{as } \varepsilon\rightarrow 0,
    \end{equation*}
    uniformly in $t\in [0,T]$. This completes the proof.    
\end{proof}

\section*{Acknowledgement}
The authors would like to thank Prof. Niyaz Tokmagambetov for his valuable comments.

\end{document}